\documentclass[twoside,11pt]{article}

\usepackage{graphicx, subfigure}
\usepackage[top=1in,bottom=1in,left=1in,right=1in]{geometry}
\usepackage[numbers,sort&compress]{natbib} 
\usepackage{amsfonts, amsmath, amssymb, amsthm, bbm}
\usepackage[normalem]{ulem}
\usepackage{comment}
\usepackage{eufrak}

\usepackage{color}
\definecolor{darkred}{RGB}{100,0,0}
\definecolor{darkgreen}{RGB}{0,100,0}
\definecolor{darkblue}{RGB}{0,0,150}

\usepackage{hyperref}
\hypersetup{colorlinks=true, linkcolor=darkred, citecolor=darkgreen, urlcolor=darkblue}
\usepackage{url}

\newtheorem{thm}{Theorem}
\newtheorem{prp}{Proposition}
\newtheorem{lem}{Lemma}
\newtheorem{cor}{Corollary}

\def\beq{\begin{equation}}
\def\eeq{\end{equation}}
\def\beqn{\begin{eqnarray*}}
\def\eeqn{\end{eqnarray*}}
\def\bitem{\begin{itemize}}
\def\eitem{\end{itemize}}
\def\benum{\begin{enumerate}}
\def\eenum{\end{enumerate}}
\def\bmult{\begin{multline*}}
\def\emult{\end{multline*}}
\def\bcenter{\begin{center}}
\def\ecenter{\end{center}}

\DeclareMathOperator*{\argmax}{argmax}
\DeclareMathOperator*{\argmin}{argmin}

\DeclareMathOperator{\tr}{tr}

\def\cA{\mathcal{A}}

\def\cC{\mathcal{C}}

\def\cG{\mathcal{G}}

\def\bbK{\bbK}

\def\cR{\mathcal{R}}

\def\bE{\mathbf{E}}

\def\bX{\mathbf{X}}

\def\bZ{\mathbf{Z}}

\def\1{{\mathbf 1}}

\def\bbK{\mathbb{K}}

\def\bbR{\mathbb{R}}

\newcommand{\E}{\operatorname{\mathbb{E}}}
\renewcommand{\P}{\operatorname{\mathbb{P}}}

\newcommand{\CORD}{\textsc{Cord}}
\newcommand{\MCORD}{\textsc{MCord}}

\newcommand{\var}[1]{\operatorname{Var}\left(#1\right)}

\newcommand{\pa}[1]{\left(#1\right)}

\def\KL{\mathrm{KL}}

\title{ PECOK:   a convex optimization approach to variable clustering }
\author{Florentina Bunea\footnote{Department of Statistical Science, Cornell University, Ithaca 15853, New York, USA}, Christophe Giraud\footnote{Laboratoire de Math\'ematiques d'Orsay, Univ.\ Paris-Sud, CNRS, Universit\'e Paris-Saclay, 91405 Orsay, FRANCE}, Martin Royer\footnote{Laboratoire de Math\'ematiques d'Orsay, Univ.\ Paris-Sud, CNRS, Universit\'e Paris-Saclay, 91405 Orsay, FRANCE}, and  Nicolas Verzelen\footnote{INRA, UMR 729 MISTEA, F-34060 Montpellier, FRANCE}}

\begin{document}
\maketitle

{\bf Abstract.}
\begin{small}
The problem of variable clustering is that of grouping  similar components of a $p$-dimensional vector $X=(X_{1},\ldots,X_{p})$, and estimating these groups from $n$ independent  copies of $X$.  When cluster similarity  is  defined via $G$-latent models, in which groups of $X$-variables have a common latent generator, and groups are relative to a partition $G$ of the index set $\{1, \ldots, p\}$,  the most natural 
clustering strategy is $K$-means. We explain why this strategy cannot lead to perfect cluster recovery and offer a correction,  based on semi-definite programing, that can be viewed as a penalized convex relaxation of $K$-means (PECOK).  We introduce a cluster separation measure 
tailored to $G$-latent models, and derive its minimax lower bound  for perfect cluster recovery. The clusters estimated by PECOK are shown to recover $G$ at a near minimax optimal cluster separation rate, a result that holds true even if $K$, the number of clusters,  is estimated adaptively from the data.  We compare PECOK with appropriate corrections of spectral clustering-type procedures, and show that the former outperforms the latter for perfect cluster recovery of minimally separated clusters. 
\end{small}

\section{Introduction}

The problem of variable clustering is that of grouping  similar components of a $p$-dimensional vector $X=(X_{1},\ldots,X_{p})$. These groups are referred to as clusters. In this work 
we investigate the problem of cluster recovery from a sample of $n$ independent copies of $X$. 
Variable clustering has had a long history in a variety of fields,  with important examples  stemming from  gene expression data \cite{GEM2Net, FreiditFrey2014, GeneExpressionCluster} or protein profile data  \cite{Bernardes2015}. The solutions to this problem are typically algorithmic and entirely data based. They include applications of $K$-means, spectral clustering, or versions of them.   The statistical 
properties of these procedures  have received a very limited amount of investigation.  It is not currently known what probabilistic cluster model on $X$ can be estimated by these popular techniques,  or  by their modifications.  Our work offers an answer to this question.  

We  study  variable clustering in  $G$-models, introduced in Bunea et al. \cite{cord},  which is a class of  models that  offer a probabilistic framework for  cluster similarity.  For a given partition $G=\{ G_{k} \}_{k=1,\ldots,K}$ of $\{1,\ldots,p\}$, the most general  of these models, called  the $G$-block covariance model, makes the mild  assumption   that permuting two variables with indices in the same  group $G_k$ of the partition does not affect the covariance $\Sigma$ of the vector $X=(X_{1},\ldots,X_{p})$. 
Hence, the off-diagonal entries of the covariance $\Sigma$ depend only on the group membership of the entries, enforcing a block-structure on $\Sigma$. This block-structure  relative to a partition $G$ can be summarized by the matrix decomposition
\begin{equation}\label{eq:block}
\Sigma=A C A^t +\Gamma,
\end{equation}
where the $p \times K$  matrix  $A$ with entries $A_{ak} : =  1_{\{a\in G_k\}}$ assigns  the index of a variable $X_a$ to a group $G_{k}$, the matrix $C$ is symmetric,  and $\Gamma$ is a diagonal matrix with $\Gamma_{aa} = \gamma_{k}$,  for all $a \in G_k$. 

We focus on an important sub-class of the $G$-block covariance models, the class of $G$-latent models,  also discussed in detail in Bunea et al.  \cite{cord}.   We say that a  zero mean vector $X$  has a latent decomposition with respect to a generic partition $G$, if 
\beq\label{eq:latent}
X_{a}=Z_{k}+E_{a},\ \textrm{for all}\ a\in G_{k},\ \textrm{and all}\ k=1,\ldots, K, 
\eeq
with $Z=(Z_{1},\ldots,Z_{K})$ a $K$-dimensional zero-mean latent vector assumed to be independent of the zero-mean error vector $E=(E_{1},\ldots,E_{p})$, which itself has independent entries, and the error variances are  equal within a group.  It is immediate to see that if  (\ref{eq:latent}) holds,  then the covariance matrix $\Sigma$ of $X$ has a $G$-block structure (\ref{eq:block}), with $C = Cov(Z)$ and $\Gamma = Cov(E)$.  Therefore, the latent $G$-models are indeed a sub-class of the $G$-covariance models, and the models are not necessarily equivalent, 
since, for instance, the matrix $C$ in (\ref{eq:block}) can be negative definite.

Intuitively, it is clear that if  the clusters of $X$  are well separated, they will be easy to estimate accurately, irrespective of the model used to define them.  This motivates the need for introducing  metrics for  cluster separation that are tailored to $G$-models,  and for  investigating the quality of cluster estimation methods  relative to the size of cluster separation.  In what follows, we refer to the $X$-variables with indices in the same group $G_k$ of a partition $G$ given by either (\ref{eq:block}) or  (\ref{eq:latent}) as a cluster. 

 When the $G$-latent  model (\ref{eq:latent}) holds,  two clusters are separated if they have different generators. Therefore,  separation between clusters  can be measured in terms 
of the canonical  "within-between group" covariance gap
\beq \label{definition_delta_distance}
\Delta(C):= \min_{j<k}\left(C_{kk}+C_{jj}-2C_{jk} \right)=\min_{j<k}\bE\left[(Z_{j}-Z_{k})^2\right], 
\eeq
since $\Delta(C) = 0$ implies $Z_j = Z_k$ a.s.

When the  $G$-block covariance model (\ref{eq:block}) holds,  the canonical separation metric between clusters 
has already been considered in \cite{cord} and is given by $$\MCORD(\Sigma):=\min_{a\stackrel{G}{\nsim} b}\max_{c\neq a,b}|\Sigma_{ac}-\Sigma_{bc}|. $$
The two metrics are connected via the following chain of inequalities, valid as soon as the size of the smallest cluster is larger than one: 
\begin{equation}\label{equiv}   \Delta(C) \leq 2\MCORD(\Sigma) \leq 2 \sqrt{\Delta(C)}\ \max_{k=1,\ldots,K} \sqrt{C_{kk}}. \end{equation}

The  variable clustering algorithm  \CORD\  introduced in \cite{cord} was shown to recover clusters given by  (\ref{eq:block}), and in particular by  (\ref{eq:latent}),  
as soon as 
\begin{equation}\label{cord1} \MCORD(\Sigma) \gtrsim \sqrt{\log(p)\over n}.\end{equation} 
Moreover, the rate of    $\sqrt{\log(p)/n}$ was  shown  in \cite{cord} to be the minimax optimal cluster separation size for correct cluster recovery, 
with respect to  the $\MCORD$ metric, in both   $G$-block covariance matrix  and $G$-latent models. \\

The first inequality in (\ref{equiv}) shows that if we are interested in the class of $G$-latent models 
together with their induced canonical cluster separation metric $\Delta(C)$, the  \CORD\ algorithm of \cite{cord} also guarantees correct cluster recovery  as  soon as $\Delta(C)\gtrsim \sqrt{\log(p)/n}$.  However, the second inequality in  (\ref{equiv}) suggests that $\Delta(C)$ can be, in order, as small as $ [\MCORD(\Sigma)]^2$, implying that, with respect to this metric, we could recover clusters that are closer together. This motivates a full investigation of variable clustering in $G$-latent models (\ref{eq:latent}),  relative to the  $\Delta(C)$ cluster separation metric, as outlined below.

\subsection{Our contribution}
We assume that the data  consist in  i.i.d. observations $X^{(1)},\ldots,X^{(n)}$ of a random vector $X$ with mean 0 and covariance matrix $\Sigma$, for which   (\ref{eq:latent}) holds relative to a partition  $G$. Our work  is 
devoted to the development of a computationally feasible method that yields an estimate 
 $\widehat{G}$ of $G$, such that  $\widehat{G} = G$, with high probability, when  $\Delta(C)$ is as small as possible, and to the  characterization of the  minimal value  of $\Delta(C)$, from a minimax perspective.

We  begin by highlighting  our main results. For simplicity, we discuss here the case where the $K$ clusters have a similar size, so that the size of the smaller cluster is $m\approx p/K$. We refer to Section~\ref{sec:convex} for the general case. 
When $X$ is Gaussian, Theorem \ref{prp:minimax_lower_bound}  below  shows that  no algorithm can estimate $G$ correctly,  with high probability,  when the latent model (\ref{eq:latent}) holds with  $ C$ fulfilling
$$\Delta(C) \lesssim |\Gamma|_{\infty} \left(\sqrt{\log(p)\over nm}\bigvee {\log(p)\over n}\right). $$
This result shows that the minimax optimal value for exact variable clustering according to model (\ref{eq:latent}),  and with respect to the metric  $\Delta(C)$,  can be much smaller than the above-mentioned $\sqrt{{\log(p)}/{n}}$ and   that the threshold for $\Delta(C)$ is  sensitive to the size of $m$:  As $m$ increases 
 the clustering problem becomes easier, in that  a smaller degree of cluster  separation is needed for exact  recovery. This property is in contrast with the fact that the  $\MCORD$ metric is not affected by the size of $m$. 
 %This has been shown in Theorem 2 in Bunea et al. \cite{cord}, where  the minimax separation rate $\MCORD(\Sigma)\gtrsim \sqrt{\log(p)/n}$ was proved to be independent of $m$ by showing that it is necessary for exact cluster recovery 
%in many different settings: some where  the number $K$ of clusters is $K=2$ or $K=p/2$;
%some where  the size $m$ of the smallest cluster is $m=2$,  or $m=p/K$.\\

 Our main result, Theorem \ref{thm:consistency} of  Section \ref{sec:convex}, shows that perfect cluster recovery  is possible,  with high probability,  via  the polynomial-time  PECOK algorithm outlined below,  when 
 %as soon as $\log(p)/n$ is small enough and 
\beq\label{condition:intro}
\Delta(C) \gtrsim |\Gamma|_{\infty} \left(\sqrt{\log(p)\vee K\over nm}\bigvee {\log(p)\vee K\over n}\right).
\eeq
  PECOK  is therefore minimax optimal as long as the number $K$ of clusters is bounded from above by $\log(p)$, and nearly  minimax optimal otherwise.   To describe  our procedure, we  begin by defining the block matrix $B$ with entries 
%With a view towards the estimation procedure, we replace the estimation target $G^*$   by the 
% matrix 
\beq\label{B:matrix}
B_{ab}= \begin{cases} {1\over |G_{k}|} & \textrm{if $a$ and $b$ are in the same group $G_{k}$,}\\
0 & \textrm{if $a$ and $b$ are in a different group.}
\end{cases}
\eeq
The groups in  a partition $G$  are in a one-to-one correspondence with the non-zero blocks of $B$.
Our PECOK algorithm has three steps, and the main step 2 produces an estimator $\widehat B$ of $B$ from which we derive the estimated partition $\widehat G$. 
The three steps of PECOK are:
\begin{enumerate}
\item Compute an estimator $\widehat \Gamma$ of the matrix $\Gamma$.
\item Solve the semi-definite program (SDP)
\begin{equation}\label{SDP1}
\widehat B=\argmax _{B \in \mathcal{C}}\langle \widehat{\Sigma} - \widehat{\Gamma}, B\rangle, \end{equation}
where  $\widehat \Sigma$ the empirical covariance matrix and
\beq\label{eq:domain}
\mathcal{C}:=\left\{ B \in \bbR^{p\times p}:
                \begin{array}{l}
                  \bullet\ B  \succcurlyeq 0 \  \ \text{(symmetric and positive semidefinite)} \\
                  \bullet\  \sum_a B_{ab} = 1,\ \forall b\\
		\bullet\ B_{ab}\geq 0,\ \forall a,b\\
		\bullet\ \tr(B) = K
                \end{array}
              \right\}.  
  \eeq
\item Compute $\widehat G$ by applying a clustering algorithm on the rows (or equivalently columns) of $\widehat B$.
\end{enumerate}
The construction of an  accurate estimator $\widehat \Gamma$ of $\Gamma$  is a  crucial step for guaranteeing the statistical optimality of  the PECOK estimator.  Estimating $\Gamma$ before estimating the partition  itself is a non-trivial task, and needs to be done with extreme care.  We devote 
Section \ref{sec:two_steps} below to the construction of an estimator $\widehat{\Gamma}$ for which the results of Theorem  \ref{thm:consistency}  hold. 

Sections  \ref{sec:convex_relaxation_Kmeans} and \ref{sec:subsect_perfect} are devoted to the second step. In Section \ref{sec:convex_relaxation_Kmeans}  we motivate the SDP (\ref{SDP1}) for estimating $B$, and present its analysis in Section \ref{sec:two_steps}, proving that $\widehat B=B$ when (\ref{condition:intro}) holds.  The required inputs for Step 2 of our algorithm are:  
(i)  $\widehat{\Sigma}$,  
the sample covariance matrix; (ii) $\widehat{\Gamma}$, the estimator produced at Step 1; and (iii) $K$, the number of groups.   When  $K$  is not known,  we offer a procedure for selecting it in a data adaptive fashion  in Section \ref{sec:adaptaption} below. Theorem \ref{thm:Kselection} of this section  shows that the resulting estimator enjoys the same properties as $\widehat{B}$ given by   (\ref{SDP1}), under the same conditions. The construction of $\widehat{B}$, when  $K$ is either known, or estimated from the data, requires  solving  the SDP (\ref{SDP1}) or a variant of it, over the convex domain $\mathcal{C}$.

  Finally, in Step 3,  we recover  the estimated partition $\widehat{G}$ from $\widehat{B}$ by  using any clustering  method  that employs the rows or, equivalently, columns, 
of $\widehat{B}$ as input. This step is done at no additional accuracy cost, as shown in Corollary \ref{cor:consistency_2_steps2},  Section \ref{sec:subsect_perfect} below. 

We summarize our three-fold contributions below. 

\begin{itemize}
\item[(i)] {\bf Minimax lower bounds.}  In Theorem \ref{prp:minimax_lower_bound} we establish minimax limits on the size of the $\Delta(C)$-cluster separation metric, for exact partition recovery,  over  the class of identifiable  $G$-latent variable models. The bounds indicate  that the difficulty of the problem decreases not only  when the sample size $n$ is large, but also, as expected, when $m$, the size of the smallest cluster, is large. 

 To the best of our knowledge,  this is  the first result of this nature for clustering via 
$G$-latent variable models. Our results  can be contrasted with the minimax cluster separation rates in \cite{cord} with respect to the $\MCORD$ metric.  They can also be contrasted with results regarding clustering of $p$ variables from data of a  different nature, network data, based on a different model,  the Stochastic Block Model (SBM). We elaborate on this point  in the remark below.

\item [(ii)] {\bf Near minimax-optimal procedure.}  We introduce and analyze PECOK, a new variable clustering procedure based on Semi-Definite Programing (SDP) for variable clustering,  that can be used  either when the number of clusters, $K$, is  known, or when it is unknown, in which case we estimate it. We prove that, in either case,  the resulting partition  estimator  $\widehat{G}$ recovers $G$, with high probability, at a near-optimal $\Delta(C)$-cluster separation rate. 
\item[(iii)] {\bf  Perfect partition recovery requires corrections of $K$-means  or spectral clustering.}  We view PECOK as a correction 
of existing clustering strategies, the correction being tailored to $G$-models.  We use the connection between  $K$-means clustering and SDP outlined in Peng and Lei \cite{PengWei07} to explain  
why $K$-means cannot perfectly recover $G$ in latent models, and present the details in Section \ref{sec:convex_relaxation_Kmeans}. 
To the best of our knowledge, this is the first work that addresses the statistical properties of corrections of $K$-means for variable clustering. 
 Moreover, we  
connect PECOK with another popular algorithm, spectral clustering, and explain why the latter
cannot be directly used, in general,  for perfect clustering in $G$-models. The details are presented in Section \ref{sec:spectral}. 
%{\color{blue} 
%\sout{
%The sub-optimal behavior of uncorrected  spectral clustering for perfect cluster recovery has already been noted in other contexts, chiefly for community detection in network analysis, when communities correspond to a 
%stochastic block model. {\color{red}"Give bibliography"}  indicated that corrections are needed, and offered suggestions tailored to  network data that are not directly transferable to our problem.  Thus, our work can also be  viewed as the first contribution 
%of this type for the problem of variable clustering, with PECOK as a  correction of spectral clustering  that achieves 
%perfect  cluster recovery at the near minimax optimal cluster separation rate, in $G$-models.}
%}
%As a benchmark, we compare theoretically and numerically the performance of our procedure
% to the popular spectral clustering algorithm,  adapted to our problem. 
\end{itemize}

\paragraph{Remark.} 
Variable clustering from {\it network data} via  Stochastic Block Models (SBM) has received a large amount of attention 
in the past years.  We contrast our contribution with that made in the SBM  literature, for instance in  \cite{guedon2014community,LeiRinaldo,chen2014statistical,lei2014generic,abbe2015community,mossel2014consistency,le2014optimization}. Although  general strategies such as SDP and  spectral clustering methods are also  employed for  cluster  estimation  from network data, and  share some similarities with  those analyzed in this manuscript, there are important differences. The most important difference stems from the nature of the data:  the data analyzed via SBM is a $p\times p$ binary  matrix, called the adjancency matrix,  with entries assumed to have been generated as independent Bernoulli random variables. In contrast,  the data matrix $\bX$ generated from a G-latent model is a $n\times p$ matrix with real entries, and rows viewed as i.i.d copies of a $p$-dimensional vector with dependent entries. In the SBM literature,  SDP-type  and spectral clustering procedures are directly applied to the adjacency matrix, whereas we need to apply them to the empirical covariance matrix $\widehat{\Sigma}:=\bX^t \bX/n$, not directly on the observed ${\bf X}$. This  has important repercussions  on the statistical analysis of the cluster estimates.  Contrary to the SBM framework, $\widehat{\Sigma}$ does not simply decompose as the sum between  the clustering signal ``$A C A^t$'' \eqref{eq:block} and a  noise component, but its decomposition also  contains cross-product terms.  The analysis of these additional terms  is non-standard,  and needs to be done with care, as illustrated by the proof of  our Theorem \ref{thm:consistency}. Moreover, in contrast to procedures tailored to  the SBM underlying model, SDP and spectral methods for $G$-latent models need to be corrected in a non-trivial fashion, as mentioned in (iii) above.

\subsection{Organization of the paper.}

In Section \ref{sec:model} we give conditions for partition identifiability in $G$-latent models. We also introduce the notation used throughout the paper.  In Section \ref{sec:convex_relaxation_Kmeans} we present the connections 
between $K$-means and PECOK. In Section \ref{sec:two_steps} we prove that one can construct an estimator of $\Gamma$, with $\sqrt{\log p/n}$ accuracy in supremum norm, before estimating the partition $G$. 
In Section \ref{sec:subsect_perfect}, Theorem \ref{thm:consistency}  and Corollary \ref{cor:consistency_2_steps2}
show  that PECOK can recover $G$ exactly, with high probability, as long as $\Delta(C)$ is sufficiently large and the number of groups $K$ is known. In Section \ref{sec:adaptaption}, Theorem \ref{thm:Kselection} and Corollary \ref{cor:Kselection} show that the same results hold, under the same conditions, when Step 2 of the PECOK algorithm 
is modified to include a step that allows  for the data dependent  selection of $K$.  Theorem \ref{prp:minimax_lower_bound} of Section \ref{sec:minimax}
gives  the  minimax lower bound for the $\Delta(C)$  cluster separation for perfect recovery. 
In Section \ref{sec:spectral} we give the connections between PECOK and spectral clustering. Theorem \ref{prp:spectral_clustering} gives sufficient conditions under which spectral clustering 
recovers partially the target partition in $G$-models, as a function of a given misclassification proportion.    We present extensions in Section \ref{sec:discussion}. All our proofs are collected in Section \ref{sec:proofs} and the Appendix.

\section{The $G$-latent variable model} \label{sec:model}

\subsection{Model and Identifiability}

 We begin by observing that   as soon as  the latent decomposition (\ref{eq:latent}) holds for $G$ it also holds for a sub-partition of $G$.  It is natural therefore  to seek the smallest of such partitions, that is the partition $G^*$ with the least number of groups for which (\ref{eq:latent}) holds.  The smallest partition is defined with respect to a partial order on sets,  and one can have multiple minimal partitions.  If a $G$-model holds with respect to a unique minimal partition $G^*$, we call the partition identifiable. We present below sufficient conditions for identifiability that do not require 
 any distributional assumption for $X$.   We assume that  $X$ is a centered random variable with covariance matrix $\Sigma$.  If $X$  follows the latent  decomposition (\ref{eq:latent}) with respect to $G^*$, we recall that    
\begin{equation}\label{eq:block1}
\Sigma=A C^* A^t +\Gamma. 
\end{equation}
with $C^*=cov(Z)$, a semi-positive definite matrix, and $\Gamma$ a diagonal matrix.  We also assume that the size $m$ of the smallest cluster is larger than 1. Since $C^*$ is positive semi-definite, we always have
$$\Delta(C^*)=\min_{j<k} (e_{j}-e_{k})^tC^*(e_{j}-e_{k})\geq 2 \lambda_{K}(C^*)\geq 0.$$
Lemma  \ref{lem:identifiable}  below shows that  requiring $\Delta(C^*)\neq 0$  ensures that the latent  decomposition (\ref{eq:latent}) holds with respect to a unique 
 partition $G^*$. 
\begin{lem}\label{lem:identifiable}
If the latent decomposition (\ref{eq:latent}) holds with $m>1$ and $\Delta(C^*)>0$, then the partition $G^*$ is identifiable.
\end{lem}

We remark that when $m = 1$, the partition may not be identifiable, and we  refer to   \cite{cord} for a counterexample. We also remark that when $\Delta(C^*)=0$ and 
  $\Gamma=\gamma I$, the partition $G^*$ is not identifiable. Therefore, the sufficient conditions for identifiability given by Lemma~\ref{lem:identifiable}  are almost necessary, and we refer to Section~\ref{sec:discussion} for further discussion. \\

In the remaining of the paper, we assume that we observe $n$ i.i.d. realizations $X^{(1)},\ldots,X^{(n)}$ of a vector $X$ following the latent  decomposition (\ref{eq:latent}) with $m>1$ and $\Delta(C^*)>0$, so that the partition $G^*$ is identifiable. We will also assume that $X \sim N(0, \Sigma)$. We refer to Section~\ref{sec:discussion} for the sub-Gaussian case.

\subsection{Notation} 
In the sequel, $\bX$,  $\bE$, and $\bZ$ respectively refer to the $n\times p$ (or $n\times K$ for $\bZ$)  matrices obtained by stacking in rows the realizations  $X^{(i)}$, $E^{(i)}$ and $Z^{(i)}$, for $i=1,\ldots,n$. We use $M_{:a} $, $M_{b:}$,  to denote the $a$-th column or, respectively,  $b$-th row of  a generic matrix $M$. 
The sample covariance matrix $\widehat{\Sigma}$ is defined by 
\[ \widehat{\Sigma} = \frac{1}{n} {\bf X}^{t}{\bf X}={1\over n}\sum_{i=1}^nX^{(i)}(X^{(i)})^t.  \]

Given a vector $v$ and $p\geq 1$, $|v|_{p}$ stands for its $\ell_p$ norm. Similarly $|A|_p$ refers to the entry-wise $\ell_p$ norm. Given a matrix $A$, $\|A\|_{op}$ is its operator norm and $\|A\|_F$ refers to the Frobenius norm. 
The bracket $\langle.,.\rangle$ refers to the Frobenius scalar product. Given a matrix $A$, denote $\mathrm{supp}(A)$ its support, that is the set of indices $(i,j)$ such that $A_{ij}\neq 0$.  We use $[p]$ to denote the set $\{1, \ldots, p\}$ and $I$ to denote the identity matrix. 

We define the variation semi-norm of a diagonal matrix $D$ as $|D|_{V}:=\max_{a} D_{aa}-\min_{a} D_{aa}$.  We use $B \succcurlyeq 0$ to denote a symmetric and positive semidefinite matrix. 
We use the notation $a\stackrel{G}{\sim}b$ whenever $a, b \in G_{k}$, for the same $k$. Also,  $m=\min_k |G_k|$ stands for the size of the smallest group.

The notation $\gtrsim$ and $\lesssim$ is used for whenever the inequalities hold up to multiplicative numerical constants.

\section{Perfect clustering with PECOK}\label{sec:convex}

\subsection{ A convex relaxation of penalized $K$-means } \label{sec:convex_relaxation_Kmeans}

For the remaining of the paper we will discuss the estimation of the identifiable partition $G^*$ discussed above. 
Knowing $G^*$ is equivalent with knowing whether  ``$a$ and $b$ are in the same group'' or ``$a$ and $b$ are in different groups", which is encoded by  the normalized partnership matrix $B^*$ given by  (\ref{B:matrix}).  In what follows, we  provide a constructive representation of $B^*$, that holds when  $\Delta(C^*) > 0$, and that can be used as the basis of an estimation procedure. 

To motivate this representation, we begin by noting that the most natural variable clustering strategy, when the  $G$-latent model (\ref{eq:latent}) holds, would be  $K$-means \cite{Lloyd}, when $K$ is known. The estimator offered  by the standard $K$-means algorithm is 
$$ \widehat{G} \in\mathop{\textrm{argmin}}_{G}\text{crit}({\bf X},G)\quad \textrm{with}\quad \text{crit}({\bf X},G)=\sum_{a=1}^p\min_{k=1,\ldots,K}\|{\bf X}_{:a}-\bar {\bf X}_{G_{k}}\|^2,$$
and $\bar {\bf X}_{G_{k}}=|G_{k}|^{-1} \sum_{a\in G_{k}} {\bf X}_{:a}$. 
Theorem 2.2 in Peng and Wei \cite{PengWei07} shows that solving the $K$-means problem is equivalent to finding the global maximum 
\begin{equation}\label{KSDP1}
\widehat B=\argmax_{B \in \mathcal{D}} \langle \widehat \Sigma,B\rangle \end{equation}
for $\mathcal{D}$ given by 
\begin{equation}\label{D}
\mathcal{D}:=\left\{ B \in \bbR^{p\times p}:
                \begin{array}{ll}
                  \bullet\ B \succcurlyeq 0 \\
                  \bullet\  \sum_a B_{ab} = 1,\ \forall b\\
		\bullet\ B_{ab}\geq 0,\ \forall a,b\\
		\bullet\ \tr(B) = K \\
                  \bullet \  B^2 = B
                \end{array}
              \right\},
\end{equation}
and then recovering $\widehat{G}$ from $\widehat{B}$.    However,  we show below that we cannot expect the $K$-means estimator  $\widehat{B}$ given by (\ref{KSDP1})  to equal $B^*$, with high probability, unless additional conditions are met. This stems from the fact that $B^*$ does not equal   $\underset{B  \in \mathcal D}{\argmax} \  \langle {\Sigma} , B \rangle $  under the  identifiability condition $\Delta(C^*) > 0$, but rather under the stronger condition  (\ref{condition:pop}) below,  which is shown in Proposition \ref{prp:pop_algo}  to be  sufficient. 
 \begin{prp}\label{prp:pop_algo} 
Assume model (\ref{eq:latent}) holds.  If 
 \begin{equation}\label{condition:pop}
 \Delta(C^*)>   \frac{2}{m}|\Gamma|_{V} ,
 \end{equation}
 then
 \begin{equation} \label{kmeans}B^* =  \underset{B  \in \mathcal D}{\argmax} \  \langle {\Sigma} , B \rangle.\end{equation}
  \end{prp}

\noindent Proposition \ref{prp:counter_example} shows that, moreover,  Condition  (\ref{condition:pop})  is needed. 

 \begin{prp}\label{prp:counter_example}
 Consider the model \eqref{eq:block} with \[C^*=\left[{\scriptsize \begin{array}{ccc} 
    \alpha & 0 & 0\\
    0 & \beta & \beta-\tau\\
    0 & \beta-\tau & \beta
   \end{array}}\right]\ , \quad \quad 
\Gamma=\left[{\scriptsize\begin{array}{ccc}
\gamma_+&0&0\\ 0& \gamma_-&0 \\ 0& 0& \gamma_-
                 \end{array}}\right],\quad \text{ and \ }|G_1|=|G_1|=|G_3|=m\ .\]
The population maximizer  $B_{\Sigma}=\argmax_{B\in\mathcal{D}} \langle  \Sigma, B\rangle$ is not equal to $B^*$ as soon as 
\[2\tau =\Delta(C^*)< \frac{2}{m}|\Gamma|_{V}\ .\] 
 \end{prp}

 \begin{cor}\label{represent} 
Assume model (\ref{eq:latent}) holds with  $\Delta(C^*) > 0$. Then
  \begin{equation} \label{correct} 
 B^*= \underset{B  \in \mathcal D}{\argmax} \  \langle \Sigma-\Gamma , B \rangle.\end{equation} 
 \end{cor}

Propositions \ref{prp:pop_algo}  and \ref{prp:counter_example} therefore show that the $K$-means algorithm does not have the capability of estimating $B^*$, unless $\Gamma= \gamma I$, whereas  Corollary  \ref{represent} suggests that a correction of the type
 \begin{equation}\label{interim}
\widetilde{B} = \underset{B  \in \mathcal D}{\argmax} \ \left\{ \langle \widehat{ \Sigma}, B \rangle - \langle   \widehat{\Gamma}, B \rangle \right \} .\end{equation} 
might be successful.  In light  of (\ref{KSDP1}),  we can view this correction as a penalization of the $K$-means criterion.

There are two difficulties with this estimation strategy.  The first one regards the construction of the estimator $\widehat{\Gamma}$  of $\Gamma$. Although, superficially, this may appear to be a simple problem, recall that we 
do {\it not}  know the partition $G^*$, in which case the problem would, indeed,  be trivial.  Instead,  we need to estimate $\Gamma$ exactly for the purpose of estimating $G^*$. We show how this vicious circle can be broken in a satisfactory manner in Section~\ref{sec:two_steps}  below. 

The second difficulty regards the optimization problem (\ref{interim}): although the objective function is linear, $\mathcal{D}$ is not convex. Following  Peng and Wei \cite{PengWei07}, we consider its  convex relaxation $\mathcal{C}$ given in (\ref{eq:domain}) above,  in which we  drop the constraint $B^2=B$.
This leads to our proposed 
estimator announced in (\ref{SDP1}), the {\bf PE}nalized {\bf CO}nvex relaxation of {\bf K}-means ({\bf PECOK}) summarized below: \medskip

\centerline{\fbox{  
\begin{minipage}{0.9\textwidth}
{\bf PECOK algorithm}
\begin{align*} 
& \text{Step 1.  Estimate} \  \Gamma \ \text{by} \  \widehat{\Gamma}. \nonumber \\
& \text {Step 2.  Estimate} \ B^* \, \text{ by} \  
\widehat B=\argmax _{B \in \mathcal{C}}\langle \widehat{\Sigma} - \widehat{\Gamma}, B\rangle. \nonumber \\
& \text{Step 3.  Estimate} \ G^* \, \text{ by  applying a clustering algorithm to the columns of }\,   \widehat{B}. \nonumber 
 \end{align*}
\end{minipage}
}}\medskip

\noindent Our only requirement  on the clustering algorithm applied in Step 3  is that it succeeds to recover the partition $G^*$ when applied to $B^*$. The standard $K$-means algorithm \cite{Lloyd} seeded with $K$ distinct centroids, kmeans++ \cite{kmeans++}, or any  approximate $K$-means as defined in (\ref{eq:definition_alpha_approximation}) in Section~\ref{sec:spectral}, fulfill this property. \\

\noindent We also have: 
 \begin{prp}\label{prp:pop_algoC} 
Assume model (\ref{eq:latent}) holds.  
 \begin{equation}\label{condition:popC}
\text{If} \  \ \ \ \Delta(C^*)>   \frac{2}{m}|\Gamma|_{V} , \ \text{ then} \ B^* =  \underset{B  \in \mathcal C}{\argmax} \  \langle {\Sigma} , B \rangle.\
 \end{equation}
In particular, when  $\Delta(C^*) > 0 $, 
 \begin{equation} \label{correctconvex} 
 B^*= \underset{B  \in \mathcal C}{\argmax} \  \langle \Sigma-\Gamma , B \rangle.\end{equation} 
  \end{prp}

Proposition  \ref{prp:pop_algoC} shows that Proposition \ref{prp:pop_algo}  and Corollary \ref{represent} continue to hold  when  the non-convex set $\mathcal{D}$ is replaced by the convex set $\mathcal{C}$, and we notice that the counterexample of Proposition \ref{prp:counter_example} also continues to be valid. 
On the basis  Proposition  \ref{prp:pop_algoC}, we expect the  PECOK 
estimator  $\widehat{B}$  to recover $B^*$, with high probability, and we show that this is indeed the case in Section \ref{sec:subsect_perfect} below, for  the estimator $\widehat{\Gamma}$ given in the following section.

\subsection{Estimation of $\Gamma$}\label{sec:two_steps}

If the groups  $G_{k}$  of the  partition $G^*$ were  known,  we could  immediately estimate $\widehat\Gamma$ by a method of moments, and obtain an estimator  with  $\sqrt{\log(p)/(nm)}$ rate  with respect to the $\ell^\infty$ norm.  If the  groups are not known,  estimation of $\Gamma$ is  still possible, but one needs to pay a price in terms of precision. Fortunately, as explained in (\ref{eq:blabla}) below, we only need to estimate $\Gamma$ at a $\sqrt{\log(p)/n}$ rate. Specifically,  in this section, our goal is to build an estimator $\widehat \Gamma$ of $\Gamma$ which fulfills $ |\widehat\Gamma - \Gamma|_{\infty}\lesssim |\Gamma|_\infty\sqrt{\log(p)/ n}$,  with high probability. In Appendix \ref{app:gamma}  we show the construction of an intuitive estimator of $\Gamma$,  for which $ |\widehat\Gamma - \Gamma|_{\infty}\lesssim |\Sigma|_\infty^{1/2}|\Gamma|_{\infty}^{1/2}\sqrt{\log(p)/ n}$, which is  not fully satisfactory, as it would suggest that the precision of an  estimate of  $\Gamma$ depends on the size of parameters not in $\Gamma$. In order to correct this and achieve our goal we need the  somewhat subtler estimator $\widehat \Gamma$ constructed below. For any $a,b\in [p]$, define 
\beq \label{eq:definition_V}
V(a,b):=   \max_{c,d \in [p]\setminus\{a,b\}}\big|\langle \bX_{:a}-\bX_{:b}, \frac{\bX_{:c}-\bX_{:d}}{|\bX_{:c}-\bX_{:d}|_2}\rangle\big| \ ,
\eeq
with the convention $0/0=0$.  Guided by the block structure of $\Sigma$, we define 
\[ne_1(a):= \argmin_{b\in [p]\setminus\{a\}}V(a,b)\quad \text{ and }\quad ne_2(a):= \argmin_{b\in [p]\setminus\{a,ne_1(a)\}}V(a,b) ,\]
to be two  ``neighbors'' of $a$, that is  two  indices  $b_1 = ne_1(a)$ and $b_2 = ne_2(a)$ 
%such that  variables $\bX_{:b_1}$ and $\bX_{:b_2}$ 
such that the covariance $\langle \bX_{:b_i} ,\bX_{:c}-\bX_{:d}\rangle$, $i =1,2$,  is most similar to $\langle \bX_{:a} ,\bX_{:c}-\bX_{:d}\rangle$, for all variables $c$ and $d$ not equal to $a$ or $b_{i}$, $i = 1,2$.  It is expected that $ne_1(a)$ and $ne_2(a)$ belong to the same group as $a$, or at least that $(\bZ A^t)_{:a}-(\bZ A^t)_{:ne_i(a)}$ is small.
Then, the estimator  $\widehat\Gamma$, which is a diagonal matrix,  is defined by 
\beq\label{eq:estim:gamma2}
\widehat\Gamma_{aa}= {1 \over n}\langle\bX_{:a}-\bX_{:ne_{1}(a)}, \bX_{:a}-\bX_{:ne_{2}(a)}\rangle,
%{1\over 2n}\left[|\bX_{:a}-\bX_{:ne_1(a)}|_2^2+ |\bX_{:a}-\bX_{:ne_2(a)}|_2^2 - |\bX_{:ne_1(a)}-\bX_{:ne_2(a)}|_2^2\right]\quad 
\quad \text{ for $a=1,\ldots, p$.}
\eeq 
The  population version of the quantity above is of order 
$$\Gamma_{aa}+ C^*_{k(a)k(a)}+C^*_{k(ne_{1}(a))k(ne_{2}(a))}-C^*_{k(a)k(ne_{1}(a))}-C^*_{k(a)k(ne_{2}(a))}\ ,$$
where $k(b)$ stands for the group of $b$. It should therefore be of order $\Gamma_{aa}$ if the above intuition holds.  Proposition \ref{prp:control_u_gamma2} below shows that  this is indeed the case.

 \begin{prp}\label{prp:control_u_gamma2}
 There exist three numerical constants $c_1$--$c_3$  such that the following holds. Assume that $m\geq 3$ and that $\log(p)\leq c_1 n$. 
 With probability larger than $1-c_{3}/p$, the estimator $\widehat\Gamma$ defined by (\ref{eq:estim:gamma2}) satisfies
 \beq\label{eq:control_check_gamma2}
|\widehat\Gamma - \Gamma|_{V}\leq 2 |\widehat\Gamma - \Gamma|_{\infty}\leq c_2 |\Gamma|_\infty\sqrt{\frac{\log(p)}{n}}\ .
 \eeq
 \end{prp}
We remark on the fact that, even though the above proposition does not make any separation assumption between the clusters, we are still able to estimate the diagonal entries $\Gamma_{aa}$ at rate $\sqrt{\log(p)/n}$ in $\ell^\infty$ norm.

\subsection{ Perfect clustering with PECOK }\label{sec:subsect_perfect}

Whereas Lemma \ref{lem:identifiable} above guarantees that $B^*$ is identifiable when $\Delta(C^*) > 0$, a larger cluster separation 
level is needed for estimating $B^*$ consistently from noisy observations. 
 Theorem \ref{thm:consistency} below   shows that  $\widehat{B} = B^*$ with high probability whenever  $\Delta(C^*)$ is larger than the sum between $ |\widehat\Gamma -\Gamma|_{V}/m$ and terms  accounting for the variability of $\widehat\Sigma$, which are the dominant terms if $\widehat{\Gamma}$ is constructed as in (\ref{eq:estim:gamma2}) above. 
\begin{thm}\label{thm:consistency}
There exist $c_1,\ldots, c_3$ three positive constants such that the following holds. 
Let $\widehat\Gamma$ be any estimator of $\Gamma$, such that $|\widehat\Gamma-\Gamma|_{V}\leq \delta_{n,p}$ with probability $1-c_{3}/(2p)$. Then, assuming that $\log(p)\leq c_1 n$, 
and that 
\begin{equation}\label{eq:assumption2}
 \Delta(C^*) \geq c_2 \left[|\Gamma|_{\infty}\left\{\sqrt{  \frac{\log p}{mn}   }+    \sqrt{\frac{p}{nm^2}} + \frac{\log(p)}{n}+ \frac{p}{nm}\right\} +   \frac{\delta_{n,p}}{m} \right]\ ,
\end{equation}
we have
$\widehat{B} = B^*$, with probability higher than $1 - c_3/p$.
\end{thm}
 
We stress once again the fact that exact partition recovery  is 
 crucially dependent on the quality of the estimation of $\Gamma$. To make this  fact as transparent as possible, we discuss below condition (\ref{eq:assumption2}) 
in a simplified setting. The same reasoning applies in general.  When all groups $G_k$ have equal size,  so that $p=mK$ (or more generally when $p\approx mK$),  and when the number $K$ of groups is smaller than $\log(p)$, Condition  \eqref{eq:assumption2} simplifies to 
 \begin{equation}\label{eq:condition_Ckk}
  \Delta(C^*) \gtrsim |\Gamma|_{\infty}\left[\sqrt{  \frac{\log p}{mn}   }+     \frac{\log(p)}{n}\right]+ \frac{\delta_{n,p}}{m} \ .
 \end{equation}
 The first term in the right-hand side of (\ref{eq:condition_Ckk}) is of order $\sqrt{  \tfrac{\log p}{mn}   }+     \tfrac{\log(p)}{n}$. It is  shown to be minimax optimal in Theorem \ref{prp:minimax_lower_bound}  of Section \ref{sec:minimax} below.   The order of magnitude of the second term,  ${\delta_{np}}/{m}$,  depends on the size  of $|\widehat\Gamma-\Gamma|_{V}$ and can become the dominant term for poor estimates of $\Gamma$. We showcase below two cases of interest.\medskip
 
\noindent{\bf  Suboptimal cluster recovery with an uncorrected  convex relaxation of $K$-means:  $\widehat\Gamma=0$.}  Theorem \ref{thm:consistency} shows that  if we took   $\widehat\Gamma=0$ in the definition of our estimator (\ref{SDP1}),  we could only guarantee recovery of clusters with a relatively large separation. Specifically, 
 when  $\widehat\Gamma=0$,  then $\delta_{n,p} \approx |\Gamma|_{V}/m$ and   (\ref{eq:condition_Ckk})  becomes 
 $\Delta(C^*)\gtrsim |\Gamma|_{V}/m$, when $m < n$, which would be strongly sub-optimal relative to    the minimax optimal separation rate of Theorem  \ref{prp:minimax_lower_bound} below. We note that the  corresponding unpenalized  estimator 
\[ \widehat B_1=\argmax _{B \in \mathcal{C}}\langle \widehat{\Sigma},  B\rangle  \]
is a convex relaxation of $K$-means, and  would still be  computationally feasible, but not statistically optimal for variable clustering in $G$-models. 
 
  \medskip 
  
% $m^{-1}$  is sub-optimal, and cannot be compensated for even by an infinite amount of data, as shown in  Proposition NEW  \ref{prp:counter_example}. For instance, if $m = \text{constant independent of $n$}$  {\color{red} say more} \\
 
\noindent{\bf Optimal cluster  recovery with a  penalized  convex relaxation of $K$-means, when  $|\widehat\Gamma-\Gamma|_{\infty}\lesssim |\Gamma|_{\infty} \sqrt{\log(p)/n}$.}  The issues raised above can be addressed by using  the corrected estimator  PECOK  corresponding to  an estimator $\widehat\Gamma$ for which  $|\widehat\Gamma-\Gamma|_{\infty}\lesssim |\Gamma|_{\infty} \sqrt{\log(p)/n}$, such as the estimator  given in  (\ref{eq:estim:gamma2}) above. 
 Then,  as desired, 
 the second term of (\ref{eq:condition_Ckk}) becomes  small  relative to  the first term of (\ref{eq:condition_Ckk}): 
\beq\label{eq:blabla}
{\delta_{np}\over m} \lesssim |\Gamma|_{\infty} \sqrt{\log(p)\over nm^2} \leq {|\Gamma|_{\infty}\over \sqrt{m}}\left[\sqrt{  \frac{\log p}{mn}   }+     \frac{\log(p)}{n}\right], 
\eeq
 since $|D|_{V}\leq 2 |D|_{\infty}$.  \\
  
With these ingredients, we can then show that  the PECOK estimator corresponding  to  $\widehat \Gamma$ defined by (\ref{eq:estim:gamma2})  recovers the true partition, at a near-minimax optimal separation rate. 

\begin{cor}\label{cor:consistency_2_steps2}
There exist $c_1, c_2,  c_3$ three positive constants such that the following holds. 
Assuming that   $\widehat \Gamma$ is defined by (\ref{eq:estim:gamma2}), $\log(p)\leq c_1 n$, 
and that 
\begin{equation}\label{eq:assumption2_2steps2}
 \Delta(C^*) \geq c_2 |\Gamma|_{\infty}\left\{\sqrt{  \frac{\log p}{mn}   }+    \sqrt{\frac{p}{nm^2}} + \frac{\log(p)}{n}+ \frac{p}{nm}\right\}\ ,
\end{equation}
then 
$\widehat{B} = B^*$, with probability higher than $1 - c_3/p$. Moreover, $\widehat{G} = G^*$, for $\widehat{G}$ given  by Step 3 of the PECOK algorithm, with probability higher than $1 - c_3/p$.
\end{cor}

\section{Adaptation to the number of groups}\label{sec:adaptaption}

In the previous section, we assumed that the number $K$ of groups is known in advance.  In many situations, however, $K$ is not known, and we address this situation here.  The  information  on $K$ was used to build our estimator  $\widehat B$, via the constraint $\tr(B)=K$ present in the definition of $\mathcal{C}$ given in \eqref{eq:domain}.  When $K$ is not known, we drop this constraint from the definition of $\mathcal{C}$, and instead penalize the scalar product $\langle \widehat\Sigma - \widehat{\Gamma}  , B\rangle$ by the trace of $B$.
Specifically, we define the adaptive estimator $\widehat B_{adapt}$ given by
 \beq\label{eq:estimator_penalized}
\widehat{B}_{adapt} :=  \underset{B  \in \mathcal{C}_0}{\argmax} \  \langle \widehat{\Sigma}- \widehat{\Gamma} , B \rangle - \widehat{\kappa} \tr(B) \ .
 \eeq
where
\beq\label{definition_C0}
\mathcal{C}_0:=\left\{ B \in \bbR^{p\times p}:
                \begin{array}{ll}
                  \bullet\ B \text{ is in } \mathcal S^+ = \{\text{symmetric and positive semidefinite}\} \\
                  \bullet\  \sum_a B_{ab} = 1,\ \forall b\\
		\bullet\ B_{ab}\geq 0,\ \forall a,b\\
                \end{array}
              \right\},
\eeq
and $\widehat \kappa$ is a data-driven tuning parameter.
The following  theorem gives conditions on $\widehat \kappa$, $\widehat \Gamma$ and $\Delta(C^*)$ which ensure exact recovery of $B^*$.

 \begin{thm}\label{thm:Kselection}
There exist $c_1, c_2, c_3$ three positive constants such that the following holds. 
Let $\widehat\Gamma$ be any estimator of $\Gamma$, such that $|\widehat\Gamma-\Gamma|_{V}\leq \delta_{n,p}$ with probability $1-c_{3}/(3p)$. Then, assuming that $\log(p)\leq c_1 n$, 
and that 
\begin{equation}\label{eq:assumption2-bis}
 \Delta(C^*) \geq c_2 \left[|\Gamma|_{\infty}\left\{\sqrt{  \frac{\log p}{mn}   }+    \sqrt{\frac{p}{nm^2}} + \frac{\log(p)}{n}+ \frac{p}{nm}\right\} +   \frac{\delta_{n,p}}{m} \right]
\end{equation}
and that, with probability larger than $1-c_3/(3p)$
 \begin{eqnarray} \label{eq:condition_widehat_kappa}
   4|\Gamma|_{\infty}\left(\sqrt{\frac{p}{n}}+ \frac{p}{n}\right) + \delta_{n,p} < \widehat{\kappa} < \frac{m}{8}\Delta(C^*) 
 \end{eqnarray}
then 
$\widehat{B}_{adapt} = B^*$, with probability higher than $1 - c_3/p$.
 \end{thm}
 Condition \eqref{eq:condition_widehat_kappa} in the above theorem encourages us to consider the following data dependent value  of  $\widehat{\kappa}$: 
 \beq\label{eq:estim:kappa}
 \widehat{\kappa}= : 5|\widehat{\Gamma}|_{\infty}\left(\sqrt{\frac{p}{n}}+ \frac{p}{n}\right), \quad \quad  \text{ where $\widehat{\Gamma}$ is defined in \eqref{eq:estim:gamma2}}\ .
 \eeq
 
 We note that the constant 5 may not be  optimal, but further analysis  of this  constant is beyond the scope of this paper. 
 %{\color{red} remove this sentence:  In practice, we suggest $\widehat{\kappa}= : |\widehat{\Gamma}|_{\infty}(1+\epsilon)\left(2\sqrt{\frac{p}{n}}+ \frac{p}{n}\right)$ with a small $\epsilon$ relative to the of $\sqrt{\log(p)/n}$)} 
 Equipped with the estimator $\widehat \Gamma$ defined in \eqref{eq:estim:gamma2} and  $\widehat \kappa$ defined in (\ref{eq:estim:kappa}), the adaptive estimator (\ref{eq:estimator_penalized}) then fulfills the following recovering property.
 
\begin{cor}\label{cor:Kselection}
There exist $c_1,\ldots, c_3$ three positive constants such that the following holds. 
Assuming that  $\widehat \Gamma$ and $\widehat{\kappa}$ are defined by (\ref{eq:estim:gamma2}) and (\ref{eq:estim:kappa}), $\log(p)\leq c_1 n$, 
and that 
\begin{equation}\label{eq:assumption2_2steps_Kappa}
 \Delta(C^*) \geq c_2 |\Gamma|_{\infty}\left\{\sqrt{  \frac{\log p}{mn}   }+    \sqrt{\frac{p}{nm^2}} + \frac{\log(p)}{n}+ \frac{p}{nm}\right\}\ ,
\end{equation}
then we have
$\widehat{B}_{adapt} = B^*$, with probability higher than $1 - c_3/p$.
\end{cor}

We observe that the Condition \eqref{eq:assumption2_2steps_Kappa} that ensures perfect recovery of $B^*$ 
when $K$ is unknown is the same as the condition \eqref{eq:assumption2_2steps2} employed in Corollary \ref{cor:consistency_2_steps2} when  $K$ was assumed to be known. This condition is shown to be near-minimax optimal in the next section.

\section{Minimax lower bound}\label{sec:minimax}

To ease the presentation, we restrict ourselves in this section to the toy model with $C^* = \tau I_K$ and $\Gamma=I_{p}$, so that, given a partition $G$, the covariance matrix decomposes as 
\beq\label{eq:model_identite}
 \Sigma_G = A_G \big(\tau I_K \big) A_G^t +  I_p\ , 
\eeq
where $A_G$ is the assignment matrix associated to  the partition $G$.  Note that, in this case $\Delta(C^*) = 2\tau$. Define $\cG$ the class of all partitions of $\{1,\ldots, p\}$ into $K$ groups of identical size $m$, therefore  $p=mK$. In the sequel, $\P_{\Sigma_G}$ refers to the normal distribution with covariance $\Sigma_G$.

\medskip 

The minimax optimal error probability for partition recovery is  defined as: 
\beq\label{eq:definition_minimax_probability}
\overline{\mathbf{R}}^*[\tau,n,m,p] := \inf_{\hat G} \sup_{G\in \cG} \mathbb{P}_{\Sigma_G} \big[ \hat{G}\neq G\big] .
\eeq

\begin{thm}\label{prp:minimax_lower_bound} 
There exists a numerical constant $c>0$ such that the following holds.
The optimal error probability of recovery $\overline{\mathbf{R}}^*[\tau,n,m,p]$ is larger than $1/7$ as soon as
 \beq \label{eq:minimax_lower_bound}
 \Delta(C^*) = 2 \tau \leq c \left[\sqrt{\frac{\log(p)}{n(m-1)}}\bigvee  \frac{\log(p)}{n} \right]\ .
\eeq
 
\end{thm}

In view of Corollary \ref{cor:consistency_2_steps2} (see also Corollary \ref{cor:Kselection}), a sufficient condition on the size of $\tau$ under which one obtains perfect partition recovery is that $\tau$ be of order $\sqrt{\frac{\log(p)}{nm}} + \frac{\log(p)}{n}$, when the ratio between the number  $K$ of groups and $\log(p)$ is bounded from above. However, the necessary conditions \eqref{eq:minimax_lower_bound} and sufficient conditions \eqref{eq:assumption2_2steps_Kappa} scale differently with $K$,  when $K$ is large.  This discrepancy  between minimax lower bounds and the performance of  estimators obtained via convex optimization algorithms has also been pinpointed in network clustering via the  stochastic block model~\cite{chen2014statistical}.  It has been conjectured that, for large $K$, there is a gap between the statistical boundary, i.e. the minimal  cluster separation  for which a statistical method achieves perfect clustering with high probability, and the polynomial boundary, i.e. the minimal  cluster separation for which there exists a polynomial-time algorithm that  achieves  perfect clustering.  Further investigation of this gap is beyond  the scope of this paper and we refer to~\cite{chen2014statistical} for more details.

\section{A comparison between PECOK and Spectral Clustering} \label{sec:spectral}

In this section we discuss connections between the clustering methods introduced above and  spectral clustering, a method 
that has  become  popular in  network clustering.  When used for variable clustering, uncorrected spectral clustering  consists in  applying a clustering algorithm, such as $K$-means,  on the rows of the $p\times K$-matrix obtained by retaining  the $K$ leading eigenvectors of $\widehat \Sigma$. Similarly to  Section \ref{sec:convex}, we propose below a correction of this algorithm. 

First, we recall  the premise of spectral clustering, adapted to our context.  For $G^*$-block covariance models as (\ref{eq:block}),  we have $\Sigma - \Gamma = AC^*A^{t}$. Let $U$ be the $p \times K$ matrix collecting the  $K$ leading eigenvectors of $\Sigma - \Gamma$.  It has been shown, see e.g. Lemma 2.1 in Lei and Rinaldo \cite{LeiRinaldo},
that $a$ and $b$ belong to the same cluster if and only if $U_{a:} = U_{b:}$ if and only if  $[UU^{t}]_{a:} =  [UU^{t}]_{b:}  $. 

Therefore, the partition $G^*$ could be recovered from $\Sigma - \Gamma$ via any clustering algorithm applied to the rows of $U$ or $UU^{t}$, for instance by a $K$-means.  It is natural therefore to consider the possibility of estimating $G^*$ by clustering the rows of 
$\widehat{U}$, the matrix of the $K$ leading  eigenvectors of 
$\widetilde{\Sigma} := \widehat{\Sigma} - \widehat{\Gamma}$.  Since $\widehat{U}$ is an orthogonal matrix, when the clustering algorithm is rotation invariant, it is  equivalent to cluster the rows of $\widehat{U}\widehat{U}^{t}$. 
We refer to this algorithm as Corrected Spectral Clustering (CSC), as it is relative to $\widehat{\Sigma} - \widehat{\Gamma}$, not  $\widehat{\Sigma}$. The two steps of CSC are then: 
\medskip

\centerline{\fbox{
\begin{minipage}{0.9\textwidth}
{\bf CSC algorithm}
\begin{enumerate}
\item Compute $\widehat{U}$, the matrix of the $K$ leading  eigenvectors of 
$\widetilde{\Sigma} := \widehat{\Sigma} - \widehat{\Gamma}$
\item Estimate $G^*$ by clustering  the rows of $\widehat{U}$, via  an $\eta$-approximation of $K$-means, defined in (\ref{eq:definition_alpha_approximation}). 
\end{enumerate}
\end{minipage}
}}\medskip

\noindent An $\eta$-approximation of $K$-means  is defined as follows.
Let  $\eta>1$ be a given positive number. Denote $\cA_{p,K}$ the collection of membership matrices, that is  $p\times K$ binary matrices whose rows contain exactly one non-zero entry. Note that a membership matrix $A\in \cA_{p,K}$ defines a partition $G$. Given a $p\times K$ matrix  $\widehat{U}$, the membership matrix $\widehat{A}$  is said to be an $\eta$-approximation $K$-means problem on $\widehat{U}$ if there exists a $K\times K$ matrix $\widehat{Q}$ such that 
\beq\label{eq:definition_alpha_approximation}
\|\widehat{U}- \widehat{A} \widehat{Q}\|_F^2  \leq \eta \min_{A\in \cA{p,k}}\min_Q \|\widehat{U} - AQ\|_F^2\ .
\eeq
Note then that $G^*$ will be estimated by  $\widehat{G}$, the partition corresponding to $\widehat{A}$.  
An example of polynomial time approximate $K$-means algorithm is given in 
Kumar {\it et al.}~\cite{ApproxKmeans}. We show below how  CSC  relates  to our proposed PECOK  estimator.

\begin{lem}\label{CSC}
When the clustering algorithm applied at the second step of Corrected Spectral Clustering (CSC) is rotation  invariant, then CSC
 is equivalent to the following algorithm: \\

\noindent Step 1.  Find  \begin{equation}\label{SDP2}
\overline{B}=\argmax\{ \langle \widetilde \Sigma,B\rangle\ : \ tr(B)=K,\  I \succcurlyeq B \succcurlyeq 0\}.
\end{equation}
\noindent Step 2.  Estimate $G^*$ by clustering  the rows of $\overline{B}$, via  an $\eta$-approximation of $K$-means, defined in (\ref{eq:definition_alpha_approximation}). 
\end{lem} 
%The proof of this lemma reveals the fact that $\overline {B} =\widehat{U}\widehat{U}^{t}$, so CSC amounts to cluster the rows of  $\overline {B}$. 

The connection between PECOK and spectral clustering now becomes clear.
The PECOK estimator involves the calculation of 
\begin{equation}\label{SDP1bis}
\widehat B=\argmax_{B}\{ \langle \widetilde \Sigma,B\rangle \ : \ B1=1,\ B_{ab}\geq 0,\ tr(B)=K,\  B \succcurlyeq 0\}.
\end{equation}
Since the matrices $B$ involved in (\ref{SDP1bis}) are doubly stochastic, their eigenvalues are smaller than 1 and hence (\ref{SDP1bis}) is equivalent to 
\[ 
\widehat B=\argmax_{B}\{ \langle \widetilde \Sigma,B\rangle\ : \ B1=1,\ B_{ab}\geq 0,\ tr(B)=K,\  I \succcurlyeq B \succcurlyeq 0\}. \]
 Note then that 
$\overline{B}$ can be viewed as a less constrained version of $\widehat{B}$, in which $\mathcal{C}$ is replaced by 
  \[ \overline{\mathcal{C}} = \{ B:  \ tr(B)=K,\  I \succcurlyeq B \succcurlyeq 0\}, \] 
 where we have dropped the  $p(p+1)/2$ constraints  given by  $B1=1$, and $B_{ab}\geq 0$. 
 We  show in what follows that the possible 
computational gains resulting from such a strategy may result in severe losses in the theoretical guarantees for exact partition recovery.
In addition, the proof of Lemma \ref{CSC} shows that $\overline{B}=\widehat{U}\widehat{U}^{t}$ so, contrary to $\widehat{B}$, the estimator $\overline{B}$ is (almost surely) never equal to $B^*$.

% {\color{red} Should we add here a little result that says we expect that since  $B^*$ does not equal $ \argmax\{ \langle \Sigma - \Gamma,B\rangle\ : \ tr(B)=K,\  I \succcurlyeq B \succcurlyeq 0\}.$ ?  But  it has the same pattern, by the same arguments as in Lemma 1, right ?}

\medskip

To simplify the presentation, we assume in the following that all the groups have the same size $|G^*_1|=\ldots=|G^*_K|=m=p/K$. We emphasize that this information is not required by either PECOK or CSC, or in the  proof of Theorem \ref{prp:spectral_clustering} below. We only 
use it here to illustrate the issues associated with CSC in a way that is not cluttered by unnecessary notation.  We denote by $\mathcal{S}_{K}$  
the set of permutations on $\{1,\ldots,K\}$ and we denote by
$$\overline{L}(\widehat G,G^*)= \min_{\sigma\in \mathcal{S}_{K}}\sum_{k=1}^K{|G^*_{k}\setminus \widehat G_{\sigma(k)}|\over m}$$
 the sum of the ratios of miss-assigned  variables with indices in $G^*_k$.  In the previous sections, we studied perfect recovery of $G^*$, which would correspond to 
 $\overline{L}(\widehat G,G^*) = 0$, with high probability. We give below conditions under which  $\overline{L}(\widehat G,G^*) \leq \rho$,  for an appropriate quantity  $\rho < 1$,  and  we show that 
 very small values of $\rho$ require large cluster separation, possibly much larger than the minimax optimal rate.  We begin with a general  theorem pertaining to partial partition recovery by CSC, under  restrictions on the  smallest eigenvalue $\lambda_{K}(C^*)$ of $C^*$.

\medskip

\begin{thm}\label{prp:spectral_clustering}
We let  $Re(\Sigma)=tr(\Sigma)/\|\Sigma\|_{op}$ denote  the effective rank of $\Sigma$.
There exist  $c_{\eta}>0$ and $c'_{\eta}>0$ only depending  on $\eta$ and numerical constants $c_1$ and $c_2$
such that the two following bounds hold.
For any $0<\rho< 1$, if 
\begin{equation}\label{eq:spectral}
 \lambda_{K}(C^*)\geq {c'_{\eta}\sqrt{K} \|\Sigma\|_{op}\over m\sqrt{\rho}}\sqrt{\frac{Re(\Sigma)\vee \log(p)}{n}},
\end{equation}
then $\overline{L}(\widehat G,G^*)\leq \rho$,  with probability larger than $1-c_2/p$.
\end{thm}

The proof extends  the arguments of \cite{LeiRinaldo},   initially developped for clustering procedures in stochastic block models, to our context. 
Specifically, we relate the error $\overline{L}(\widehat G,G^*)$ to the noise level, quantified in this problem by  $\|\widetilde{\Sigma}-AC^*A^t\|_{op}$.  We then employ the results of 
 \cite{koltchinskii2014concentration} and \cite{BuneaEtal} to show that this 
operator norm can be controlled, with high probability, which leads to the conclusion of the theorem.  \\

We observe that $\Delta(C^*)\geq 2 \lambda_{K}(C^*)$, so the lower bound  (\ref{eq:spectral}) on $ \lambda_{K}(C^*)$ enforces the same lower-bound on $\Delta(C^*)$.
To further facilitate the comparison with the performances of PECOK,  we discuss both the conditions and the conclusion of 
this theorem  in the simple setting where $C^*=\tau I$ and $\Gamma=I$. Then, the cluster separation measures coincide up to a factor 2,  $\Delta (C^*) = 2\lambda_K(C^*) = 2\tau$.

\begin{cor}[Illustrative example: $C^*=\tau I$ and $\Gamma=I$]\label{cor1_spectral} There exist three positive numerical constants $c_{1,\eta}$, $c_{2,\eta}$ and $c_3$ such that the following holds.
 For any $0<\rho<1$, if 
\beq\label{eq:condition_consistency_low_rank}
 \rho\geq c_{1,\eta}\Big[ \frac{K^2}{n}+ \frac{K\log(p)}{n}\Big]\quad \quad \text{ and }\quad \quad \tau \geq c_{2,\eta}\Big[ {K^2\over \rho n} \vee \frac{K}{\sqrt{\rho nm}}\Big]\ ,
 \eeq
then $\overline{L}(\widehat G,G^*)\leq \rho$, with probability larger than $1-c_3/p$.~\\
\end{cor}

Recall that, as a benchmark, Corollary \ref{cor:consistency_2_steps2} above states that, when $\widehat{G}$ is obtained via the PECOK algorithm, and  if 
\beq \label{eq:consistency_SDP_C=I}
 \tau \gtrsim \sqrt{  \frac{K\vee \log p }{mn}   }+     \frac{\log(p)\vee K }{n}\ ,
\eeq
then  $\overline{L}(\widehat G,G^*) = 0$, or equivalently, $\widehat{G} = G^*$, with high probability. We can  therefore provide the following summary, for the simple case $C^*= \tau I$. \\

\noindent{\bf Summary: PECOK vs CSC when $C^*=\tau I$.}\\

\noindent {\bf  1.  $\rho$ is a user specified small value, independent of $n$ or $p$,  and the 
number of groups $K$ is either a constant or grows at most as $\log p$. } 
In this case,  the size of the cluster separation given by either  Condition \eqref{eq:condition_consistency_low_rank} and \eqref{eq:consistency_SDP_C=I}  are essentially the same, up to unavoidable $\log p$ factors. The difference is that, in this regime, CSC guarantees 
recovery up to a fixed, small,  fraction of mistakes, whereas PECOK guarantees exact recovery. \\

 \noindent {\bf  2.  $\rho \rightarrow 0 $. }
 Although perfect recovery, with high probability,  cannot be guaranteed for CSC, we could be close to it by requiring $\rho$ to be close to zero.  In this case,  the distinctions between Conditions \eqref{eq:condition_consistency_low_rank} and \eqref{eq:consistency_SDP_C=I} become much more pronounced. Notice that whereas the latter condition  is independent of $\rho$,  in the former there is a trade-off between the precision $\rho$ and the size of the 
cluster separation. Condition  \eqref{eq:consistency_SDP_C=I}  is the near-optimal separation condition that guarantees that $\rho = 0$ when PECOK is used.  However, if in \eqref{eq:condition_consistency_low_rank} we took, for instance,  $\rho$  to be proportional to  $K^2/n$,  whenever the latter is small, the cluster separation 
requirement  for CSC would become 
\[ \tau \gtrsim 1, \]
which is unfortunately  very far from the optimal minimax rate.  \medskip

The phenomena summarized above have already been observed in the analysis of spectral clustering algorithms for network clustering via the Stochastic Block Model (SBM), for instance in   \cite{LeiRinaldo}. When we move away from the case $C^*=\tau I$ discussed above, the sufficient condition \eqref{eq:spectral} of the general Theorem \ref{prp:spectral_clustering} for CSC compares unfavorably  with condition \eqref{eq:assumption2_2steps2}
of Corollary \ref{cor:consistency_2_steps2} for PECOK even when $\rho$ is a fixed value. 

  For instance, consider  $C^*=\tau I+\alpha J$, with $J$ being the matrix with all entries equal to one, and  $\Gamma=I$. Notice  that in this case we continue to have $\Delta(C^*)= 2\lambda_{K}(C^*)= 2\tau$. Then, for a given, fixed, value of $\rho$ 
 and $K$ fixed, condition \eqref{eq:spectral} of the general Theorem \ref{prp:spectral_clustering} guarantees 
   $\rho$-approximately correct clustering via CSC for  the cluster separation 
\[\tau \gtrsim \frac{\alpha\sqrt{\log(p)}}{\sqrt{n\rho}}\ ,\]
which is  independent of $m$, unlike the minimax cluster separation rate that we established in Theorem \ref{prp:minimax_lower_bound} above. Although we only compare sufficient conditions for partition estimation, this phenomenon further supports the merits  the  PECOK method proposed in this work.

\begin{comment}
A number of two stage corrections have been proposed in that context, for exact recovery,  with examples   including 
~\cite{lei2014generic,abbe2015community,mossel2014consistency}. {\color{red} check BIBLIO}  These strategies have a 
crucial element in common:  the first stage  requires partial recovery at a fixed level $\rho$, at the minimax optimal cluster separation rates, for clusters defined by a general  SBM.  Our summary above 
may suggest that such strategies are also possible for variable clustering, at least when the number of clusters $K$ is small. This is however not the case, and it becomes clear when we move away from  the case $C^*=\tau I$ discussed above.  For instance, when  $C^*=\tau I+\alpha J$, with $J$ being the matrix with all entries equal to one, and  $\Gamma=I$. Notice  that in this case we continue to have $\Delta(C^*)= 2\lambda_{K}(C^*)= 2\tau$. Then, for a given, fixed, value of $\rho$ 
 and $K$ fixed, condition \eqref{eq:spectral} requires the cluster separation 
\[\tau \gtrsim \frac{\alpha\sqrt{\log(p)}}{\sqrt{n\rho}}\ ,\]
which is  independent of $m$, unlike the minimax cluster separation rate that we established in Theorem \ref{prp:minimax_lower_bound} above.  Therefore, the correction strategies employed in SBM are not directly transferable to variable clustering, which further supports the merits  our PECOK method. 
\end{comment}

\section{Extensions}\label{sec:discussion}

In this section we discuss briefly immediate generalizations of the framework presented above. 

First we note that we only assumed that $X$ is Gaussian in order to keep the notation as simple as possible.  All our arguments continue to hold if $X$ is sub-Gaussian,  in which case all the concentration inequalities used in our proofs can be obtained via appropriate applications of Hanson-Wright inequality~\cite{rudelson2013hanson}. Moreover, the bounds in operator norm between covariance matrices and their estimates continue to hold, 
at the price of an additional $\log (p)$ factor, under an extra assumption on the moments 
of $X$, as explained in section 2.1  of \cite{BuneaEtal}. 

We can also generalize slightly the modeling framework. If (\ref{eq:latent}) holds for $X$, we can alternatively assume that the error variances within a block are not equal. The implication is that 
in the decomposition of the corresponding $\Sigma$ the diagonal matrix $\Gamma$ will have arbitrary non-zero entries.  The characterization (\ref{correctconvex}) of $B^*$ that motivates PECOK  is unchanged, and so is the rest of the paper, with the added bonus that in Lemma \ref{lem:identifiable} the sufficient identifiability condition $\Delta(C^*) > 0$ also becomes necessary. 
We have preferred the set-up in which $\Gamma$ has equal diagonal entries per cluster  only to facilitate direct comparison with \cite{cord}, where this assumption is made.

\section{Proofs}\label{sec:proofs}

\subsection{Proof of Lemma~\ref{lem:identifiable}}
When (\ref{eq:latent}) holds, then we have $\Sigma=A C^* A^t+\Gamma^*$, with $A_{ak}=1_{a \in G^*_{k}}$. In particular, writing $k^*(a)$ for the integer such that $a\in G^*_{k^*(a)}$, we have 
$\Sigma_{ab}=C^*_{k^*(a)k^*(b)}$ for any $a\neq b$.

Let $a$ be any integer between 1 and $p$ and set
$$V(a)=\{a':a'\neq a,\ Cord(a,a')=0\},\quad \text{where}\quad Cord(a,a')=\max_{c\neq a,a'}|\widehat \Sigma_{ac}-\widehat \Sigma_{a'c}|.$$
We prove below that 
$V(a)=G^*_{k^*(a)}\setminus \{a\}$, and hence the partition $G^*$ is identifiable from $\Sigma$. 

First, if $a'\in  G^*_{k^*(a)}\setminus \{a\}$, then $k^*(a')=k^*(a)$, so
$$\Sigma_{ac}-\Sigma_{a'c}=C^*_{k^*(a)k^*(c)}-C^*_{k^*(a)k^*(c)}=0$$
for any $c\neq a,a'$. So $a'\in V(a)$ and hence $G^*_{k^*(a)}\setminus \{a\}\subset V(a)$.

Conversely, let us prove that $V(a) \subset G^*_{k^*(a)}\setminus \{a\}$. Assume that it is not the case, hence there exists $a'\in V(a)$ such that $k^*(a')\neq k^*(a)$. Since $m>1$ and $a'\notin G^*_{k^*(a)}$, we can find $b\in G^*_{k^*(a)}\setminus \{a\}$ and $b'\in G^*_{k^*(a')}\setminus \{a'\}$.
Since $a'\in V(a)$, we have $Cord(a,a')=0$ so
\begin{align*}
0&=\Sigma_{ab}-\Sigma_{a'b}=C_{k^*(a)k^*(a)}^*-C_{k^*(a)k^*(a')}^*\\
0&=\Sigma_{ab'}-\Sigma_{a'b'}=C^*_{k^*(a)k^*(a')}-C^*_{k^*(a')k^*(a')}.
\end{align*}
In particular, we have
$$\Delta(C^*)\leq C^*_{k^*(a)k^*(a)}+C^*_{k^*(a')k^*(a')}-2C^*_{k^*(a)k^*(a')}=0,$$
which is in contradiction with $\Delta(C^*)>0$.
%$$\Delta(C^*)=\min_{j<k} (e_{j}-e_{k})'C^*(e_{j}-e_{k})\geq 2 \lambda_{K}(C^*)>0,$$
%where the last inequality has been enforced by $C^*$ positive definite. 
So it cannot hold that $k^*(a')\neq k^*(a)$, which means that any $a'\in V(a)$ belongs to $G^*_{k^*(a)}\setminus \{a\}$, i.e. $V(a) \subset G^*_{k^*(a)}\setminus \{a\}$. This conclude the proof of the equality $V(a)=G^*_{k^*(a)}\setminus \{a\}$ and the proof of the Lemma is complete. \\

{\it In order to avoid notational clutter in the remainder of the paper, we re-denote $C^*$ by $C$ and $G^*$ by $G$.}  \\

\subsection{ Proofs of Proposition  \ref{prp:pop_algoC}, Proposition \ref{prp:pop_algo}  and Corollary \ref{represent} }

Since ${\cal D}\subset {\cal C}$, and $B^* \in \mathcal{D}$, the proofs of Proposition \ref{prp:pop_algo}  and Corollary \ref{represent} follow from the proof of Proposition  \ref{prp:pop_algoC}, given below.  The basis of the proof of Proposition  \ref{prp:pop_algoC} is the following Lemma.

\begin{lem}\label{lem:support}
 The collection $\cC$ contains only one matrix whose support is included in $\mathrm{supp}(B^*)$, that is 
 \[
  \cC\cap \big\{B,\ \mathrm{supp}(B)\subset \mathrm{supp}(B^*)\big\}= \{B^*\}\ .
 \]
\end{lem}
\begin{proof}
Consider any matrix $B\in \cC$ whose support is included in $\mathrm{supp}(B^*)$. Since $B1=1$, it follows that each submatrix $B_{G_kG_k}$ is symmetric doubly stochastic. Since  $B_{G_kG_k}$ is also positive semidefinite, we have \[tr(B_{G_kG_k})\geq \|B_{G_kG_k}\|_{op}\geq 1^tB_{G_kG_k}1/|G_k|=1\ . \]
As $B\in \cC$, we have $\tr(B)= K$, so all the submatrices $B_{G_kG_k}$ have a unit trace. Since $\|B_{G_kG_k}\|_{op}\geq 1$, this also enforces that $B_{G_kG_k}$ contains only one non-zero eigenvalue and that a corresponding eigenvector is the constant vector $1$. As a consequence, $B_{G_kG_k}=11^t/|G_k|$ for all $k=1,\ldots,K$ and $B=B^*$. 
\end{proof}

As a consequence of Lemma~\ref{lem:support}, we only need to prove that,     
\[ \langle {\Sigma}, B^*-B \rangle  > 0, \  \mbox{ for all} \  B \in \mathcal{C}\text{ such that } \ \mathrm{supp}(B)\nsubseteq \mathrm{supp}(B^*).\]
 We have 
$$ \langle {\Sigma}, B^*-B \rangle=\langle ACA^t,B^*-B\rangle + \langle \Gamma , B^*-B \rangle.$$
Define the $p$-dimensional vector $v$ by $v=diag(ACA^t)$. 
Since $B1=1$ for all $B\in\cal{C}$, we have $\langle v1^t+1v^t,B^*-B\rangle=0$. Hence, we have
\begin{align}
\langle ACA^t,B^*-B\rangle&=\langle ACA^t-{1\over 2}(v1^t+1v^t),B^*-B\rangle\nonumber\\
&= \sum_{j,k}\sum_{a\in G_{j},\,b\in G_{k}}  \left(C_{jk}-{C_{jj}+C_{kk}\over 2}\right)(B^*_{ab}-B_{ab})\nonumber\\
&= \sum_{j\neq k}\sum_{a\in G_{j},\,b\in G_{k}}  \left({C_{jj}+C_{kk}\over 2}-C_{jk}\right)B_{ab}\nonumber\\
&= \sum_{j\neq k} \left({C_{jj}+C_{kk}\over 2}-C_{jk}\right) |B_{G_jG_k}|_1,\label{eq:signalpop}
\end{align}
where $B_{G_jG_k}=[B_{ab}]_{a\in G_{j},\, b\in G_{k}}$.
Next lemma lower bounds $\langle \Gamma , B^*-B \rangle$.
\begin{lem}\label{lem:control_Gamma}
\beq\label{eq:control_gamma}
\langle \Gamma , B^*-B \rangle \geq - \frac{\max_k \gamma_k -\min_k \gamma_k}{m}\sum_{k\neq j}|B_{G_jG_k}|_1\ .
\eeq
\end{lem}

\begin{proof}[Proof of Lemma \ref{lem:control_Gamma}]
By definition of $B^*$ and since $\tr(B)=\tr(B^*)=K$, we have 
\begin{align} 
\langle \Gamma , B^*-B \rangle&=\langle \Gamma -|\Gamma|_{\infty}I, B^*-B \rangle\nonumber\\
&=  \sum_{k=1}^K (\gamma_{k}-|\Gamma|_{\infty})\Big[1- \tr(B_{G_kG_k})\Big] \nonumber\\
& \geq  - (\max_k \gamma_k -\min_k \gamma_k )\sum_{k=1}^K \Big[1- \tr(B_{G_kG_k})\Big]_{+} \label{eq:gamma}
\end{align}
Since $B_{G_kG_k}$ is positive semidefinite, we have 
\[0\leq 1^t B_{G_kG_k} 1/|G_k|\leq \|B_{G_kG_k}\|_{op}\leq \tr(B_{G_kG_k})\ .\] As a consequence, $1-\tr(B_{G_kG_k})\leq 1 - 1^t B_{G_kG_k} 1/|G_k|$. Since  $B1=1$, we conclude that 
\[
 \big[1-\tr(B_{G_kG_k})\big]_{+}\leq 1 - {1^t B_{G_kG_k} 1\over |G_k|}=\frac{1}{|G_k|}\sum_{j:j\neq k}|B_{G_jG_k}|_1.
\]
Coming back to \eqref{eq:gamma}, this gives us
\[
\langle \Gamma , B^*-B \rangle \geq - \frac{\max_k \gamma_k -\min_k \gamma_k}{m}\sum_{k\neq j}|B_{G_jG_k}|_1\ .
\]
\end{proof}

Hence, combining (\ref{eq:signalpop}) and Lemma~\ref{lem:control_Gamma}, we obtain
$$ \langle {\Sigma}, B^*-B \rangle\geq \sum_{j\neq k} \left({C_{jj}+C_{kk}\over 2}-C_{jk}- \frac{\max_k \gamma_k -\min_k \gamma_k}{m}\right) |B_{G_jG_k}|_1.$$
The condition~(\ref{condition:pop}) enforces that if $\mathrm{supp}(B)\nsubseteq \mathrm{supp}(B^*)$ then   $\langle {\Sigma}, B^*-B \rangle  > 0$. This concludes the proof of  (\ref{condition:popC}) and  (\ref{correctconvex}).
The proof of Proposition  \ref{prp:pop_algoC} is complete.

 \subsection{Proof of Proposition \ref{prp:counter_example}}
  By symmetry, we can assume that the true partition matrix $B^*$ is diagonal block constant. Define the partition matrix $B_1:= \left[{\scriptsize \begin{array}{ccc} 2/m & 0 & 0 \\
  0& 2/m & 0 \\
  0& 0 & 1/(2m)
  \end{array}}\right]$ where the first two blocks are of size $m/2$ and the the last block has size $2m$. The construction  of the  matrix $B_1$ amounts to merging  groups $G_2$ and $G_3$,  and to splitting $G_1$ into  two groups of equal size. Then, 
  \[\langle  \Sigma, B^*\rangle= \gamma_{+}+2\gamma_{-}+ m tr(C)\ , \quad \quad \langle  \Sigma, B_1\rangle= 2\gamma_+ + \gamma_-+ m tr(C)-m\tau\ .\]
  As a consequence, $\langle  \Sigma, B_1\rangle < \langle  \Sigma, B^*\rangle$ if and only if $\tau> \frac{\gamma_+- \gamma_-}{m}$.

\subsection{Proof of Theorem \ref{thm:consistency}}
As a consequence of Lemma~\ref{lem:support} page \pageref{lem:support}, when  $\widehat{B}$ is given by (\ref{SDP1}), we only need to prove that \begin{equation}\label{toshow}  \langle \widehat{\Sigma}-\widehat\Gamma, B^*-B \rangle  > 0, \  \mbox{ for all} \  B \in \mathcal{C}\text{ such that } \ \mathrm{supp}(B)\nsubseteq \mathrm{supp}(B^*), \end{equation}
 with high probability. \\

\noindent We begin by recalling the notation:  ${\bf X}$ denotes the $n \times p$ matrix of observations. 
Similarly, ${\bf Z}$ stands for the $n \times K$ matrix corresponding to the un-observed latent variables and  ${\bf E}$ denotes the $n \times p$ error matrix defined by $\bX=\bZ A^t+\bE$. \\

\noindent Our first goal is to decompose $ \widehat{\Sigma}-\widehat\Gamma$ in such a way that the distance $|\bZ_{:k}-\bZ_{:j}|_2^2$ becomes evident, as this is the empirical counter-part of  the key quantity $Var(Z_j - Z_k) =[C_{jj}+ C_{kk}-2C_{jk}]$ which informs cluster separation.  To this end, recall that  $n\widehat{\Sigma}= \bX^t \bX$ and let 
 $\widetilde\Gamma = \frac{1}{n}{\bf E^t E}$. Using the latent model representation, we further have 
 \[  n\widehat{\Sigma}       =  A {\bZ}^t \bZ A^t + n\widetilde\Gamma + A({\bf Z^t E}) + ({\bf E^t Z})A^t.  \] 
Using the fact that for any vectors $v_1$ and $ v_2$ we have $|v_1 - v_2|^2_2 = |v_1|_2^2  + |v_2|^2 - 2<v_1, v_2>$, 
we can write
\[ [ A {\bZ}^t \bZ A^t ]_{ab} =  \frac{1}{2} |[A{\bf Z}^t]_{a:}|_2^2 +  \frac{1}{2} |[A{\bf Z}^t]_{b:}|_2^2  -\frac{1}{2} |[A{\bf Z}^t]_{a:}- [A{\bf Z}^t]_{b:}|_2^2,\]
for any  $1\leq a,b\leq p$. We also observe that 

\[  [ A({\bf Z^t E}) + ({\bf E^t Z})A^t] _{ab} =  [\bE^t_{b:}- \bE^t_{a:}][( A\bZ ^t)_{a:}- (A\bZ^t )_{b:}] + [A\bZ^t\bE]_{aa} + 
 [A\bZ^t\bE]_{bb} . \]
  Define the $p\times p$ matrix $W$ by 
\beq\label{eq:definition_W}
W_{ab}:= n(\widehat{\Sigma}_{ab} -\widehat\Gamma_{ab})- \frac{1}{2} |[A{\bf Z}^t]_{a:}|_2^2 - \frac{1}{2} |[A{\bf Z}^t]_{b:}|_2^2 -  [A{\bf Z^t E}]_{aa}- [A{\bf Z^t E}]_{bb}. 
\eeq
Combining the  four displays above  we have 

\begin{equation} \label{dec} W=W_1+W_2+ n(\widetilde\Gamma-\widehat\Gamma), \end{equation}   with 
\beq\label{eq:definition_W_decomposion}
(W_{1})_{ab}:= -\frac{1}{2} |[A{\bf Z}^t]_{a:}- [A{\bf Z}^t]_{b:}|_2^2, \ \  \ \ (W_2)_{ab}:= [\bE^t_{b:}- \bE^t_{a:}][(A\bZ^t )_{a:}- (A\bZ^t)_{b:}], 
\eeq
for any $1\leq a,b\leq p$.  Observe that $W-n(\widehat{\Sigma}-\widehat\Gamma)$ is a sum of four matrices, two 
of which are of the type $1v_1^t$, and two of the type $v_21^t$, for some vectors $v_1, v_2  \in \cR^p$. Since for any two matrices $B_1$ and $B_2$ in $\cC$, we have  $B_1 1= B_21=1$, it follows that 
\[
 \langle W-n(\widehat{\Sigma}-\widehat\Gamma), B_1-B_2\rangle = 0  \ .
\]
As a consequence and using the decomposition (\ref{dec}),  proving (\ref{toshow}) reduces to proving 
\beq\label{eq:main_objective}
\langle W_1 + W_2 + n(\widetilde\Gamma-\widehat\Gamma), B^*-B \rangle  > 0,  \  \mbox{ for all} \  B \in \mathcal{C}\text{ such that } \ \mathrm{supp}(B)\nsubseteq \mathrm{supp}(B^*).
\eeq
We will analyze the inner product between $B^* - B$ with each of the three matrices in (\ref{eq:main_objective}) separately. \\

The matrix $W_{1}$ contains the information about the clusters, as we explain below. Note that for two variables $a$ and $b$ belonging to the same group, $(W_{1})_{ab}=0$ and for two variables $a$ and $b$ belonging to different groups $G_j$ and $G_k$, $(W_{1})_{ab}=-|\bZ_{:i}- \bZ_{:k}|_2^2/2$. As a consequence, $\langle W_1, B^*\rangle=0$ and 
\[
\langle W_1, B^*-B \rangle =  {1\over 2} \sum_{j\neq k}|\bZ_{:j}-\bZ_{:k}|_2^2 \sum_{a\in G_j,\ b\in G_k}B_{ab}.  \]
In the sequel, we  denote by $B_{G_j,G_k}$ the submatrix $(B_{ab})_{a\in G_j,\,b\in G_k}$. Since all the entries of $B$ are nonnegative, 
\beq\label{eq:main1}
\langle W_1, B^*-B \rangle =  {1\over 2} \sum_{j\neq k}|\bZ_{:j}-\bZ_{:k}|_2^2 |B_{G_jG_k}|_1\ . 
\eeq

\noindent We will analyze below the two remaining cross products. As we shall control the same quantities $\langle W_2,B^*-B\rangle$ and $\langle \widetilde\Gamma-\widehat\Gamma, B^*-B\rangle$ for $B$ in the  larger class $\cC_0$ given by \eqref{definition_C0} in the proof of Theorem \ref{thm:Kselection}, we state the two following lemmas for $ B  \in \cC_0$. Their proofs are given after the proof of this theorem.

\begin{lem}\label{lem:W2}
With probability larger than $1-c'_0/p$, it holds that
 \beq\label{eq:main3}
|\langle W_2, B^*-B \rangle|\leq c_1 \sqrt{\log(p)}\sum_{j\neq k}|\Gamma|_{\infty}^{1/2} |\bZ_{:j}-\bZ_{:k}|_2 |B_{G_jG_k}|_1\ ,
\eeq 
simultaneously over all matrices $B\in \cC_0$.  
\end{lem}

\noindent It remains to control the term corresponding to the empirical covariance matrix of the noise $\bE$. This is the main technical difficulty in this proof.

\begin{lem}\label{lem:covariance}
With probability larger than $1-c_0/p$, it holds that
\begin{eqnarray}\label{main2-bis}
n|\langle \widetilde\Gamma-\widehat\Gamma, B^*-B\rangle|&\leq  &c_2 \left[| \Gamma|_{\infty} \left( \sqrt{  n\frac{\log p}{m}   }  \vee    \sqrt {\frac{np}{m^2}} \vee \frac{p}{m}\right)  + {n |\widehat\Gamma-\Gamma|_{V}\over m}     \right] \sum_{j\neq k }|B_{G_jG_k}|_1\\ &&+ 4|\Gamma|_{\infty}\left[\sqrt{\frac{p}{n}}\vee \frac{p}{n}\right](\tr(B)-K)+ [\tr(B)-K]_{+}| \Gamma-\widehat\Gamma|_{V}, \nonumber
\end{eqnarray}
simultaneously over all matrices $B\in \cC_0$.  
\end{lem}
\noindent For all matrices $B\in \cC$ we have  $\tr(B)-K=0$. Therefore,   the second line of \eqref{main2-bis} is zero for 
the purpose of this proof.  Combining \eqref{eq:main1},  \eqref{eq:main3}, and \eqref{main2-bis} we obtain that,  with probability larger than $1-c/p$, 
 \begin{eqnarray}\label{main}
 \langle W, B^*-B \rangle \geq  \sum_{j\neq k} \Bigg[{1\over 2}|\bZ_{:j}-\bZ_{:k}|^2_2  -  c_1 \sqrt{\log(p)} |\Gamma|_{\infty}^{1/2}|\bZ_{:j}-\bZ_{:k}|_2 \\ -c_2 \frac{n|\widehat\Gamma-\Gamma|_{V}}{m}  - c_{2}|\Gamma|_{\infty}\left( \sqrt{  \frac{n\log p}{m}   }  \vee    \sqrt {\frac{np}{m^2}} \vee \frac{p}{m}\right )\Bigg] |B_{G_jG_k}|_1\ , \nonumber
\end{eqnarray}
simultaneously for all $B\in \cC$. Therefore, if each term in the bracket of (\ref{main}) is positive, with high probability,  (\ref{eq:main_objective}) will follow,   since any matrix $B\in \cC$ whose support is not included in $\mathrm{supp}(B^*)$ satisfies $ |B_{G_jG_k}|_1>0$ for some $j\neq k$,

Since for any $j \neq k$ the differences $Z_{ij} - Z_{ik}$, for $1 \leq i \leq n$ are i.i.d. Gaussian random variables with mean zero and variance $C_{jj}+ C_{kk}-2C_{jk}$, 
we can apply Lemma \ref{lem:quadratic} in Appendix \ref{sec:LaurentMassart} with  $A = -I$ and $t = \log p$. Then, 
if $\log(p)<n/32$,  \beq \label{eq:control_distance}
|\bZ_{:j}-\bZ_{:k}|^2_2\geq n[C_{jj}+ C_{kk}-2C_{jk}]/2 \ ,
\eeq
simultaneously for all $j\neq k$,  with probability larger than $1-1/p$.  Then, on the event for which (\ref{eq:control_distance}) holds intersected  with  the event $|\Gamma-\widehat\Gamma|_{V}\leq \delta_{n,p}$, Condition \eqref{eq:assumption2} enforces that, for all $j\neq k$,
\[  \frac{1}{4}|\bZ_{:j}-\bZ_{:k}|^2_2 \geq c_1 \log(p) |\Gamma|_{\infty}, \]
and also 
\begin{eqnarray}  \label{eq:control_distance2}
 \frac{1}{4}|\bZ_{:j}-\bZ_{:k}|^2_2 &\geq& c_2 \frac{n|\widehat\Gamma-\Gamma|_{V}}{m} \\ &+&  c_{2}|\Gamma|_{\infty}\left( \sqrt{  \frac{n\log p}{m}   }  \vee    \sqrt {\frac{np}{m^2}} \vee \frac{p}{m}\right )
\end{eqnarray}
with probability larger than $1-c'/p$. Therefore,  $ \langle W, B^*-B \rangle  > 0$, for all $B\in \cC$,  which concludes the proof of this theorem . $\blacksquare$

\begin{proof}[Proof of Lemma \ref{lem:W2}]
Consider any $a$ and $b$ in $[p]$ and let $j$ and $k$ be such that $a\in G_j$ and $b\in G_k$. If $j=k$, $(W_{2})_{ab}=0$. 
If $j\neq k$, then $(W_{2})_{ab}$ follows, conditionally to $\bZ$, a normal distribution with variance 
$|\bZ_{:j}-\bZ_{:k}|_2^2[\gamma_{j}+ \gamma_{k}]$. Applying the Gaussian concentration inequality together with the union bound, we conclude that with probability larger than $1-1/p$, 
\[|(W_2)_{ab}|\leq c_{1} |\bZ_{:j}-\bZ_{:k}|_2\sqrt{\log(p)}(\gamma_{j}^{1/2}\vee \gamma_{k}^{1/2}  )\ ,\]
simultaneously for all $a,\ b\in [p]$. It then follows that 
\[
|\langle W_2, B^*-B \rangle|\leq c_1 \sqrt{\log(p)}|\Gamma|_{\infty}^{1/2}\sum_{j\neq k}|\bZ_{:j}-\bZ_{:k}|_2 |B_{G_jG_k}|_1\ ,
\]
becauce $B^*_{ab}=0$ and $B_{a,b}\geq 0$ for all $a\in G_j$ and $b\in G_k$, with $j\neq k$.

\end{proof}

\begin{proof}[Proof of Lemma \ref{lem:covariance}] 
We split the scalar product $\langle  \widetilde\Gamma- \widehat\Gamma , B^* - B \rangle$ into two terms $\langle  \widetilde\Gamma- \widehat\Gamma , B^* - B \rangle=\langle  \widetilde\Gamma- \Gamma , B^* - B \rangle+\langle  \Gamma- \widehat\Gamma , B^* - B \rangle$. \\

\smallskip

\noindent {\it (a) Control of $\langle  \widetilde\Gamma- \Gamma , B^* - B \rangle$.}\\*

\noindent Observe first that  $B^*$ is a projection matrix that induces the following decomposition  of $\widetilde\Gamma- \Gamma$.
\begin{eqnarray*}
\widetilde\Gamma-\Gamma&=& B^*(\widetilde\Gamma-\Gamma) + (\widetilde\Gamma-\Gamma) B^*  - B^* (\widetilde\Gamma-\Gamma)B^* +  (I - B^*)(\widetilde\Gamma-\Gamma)(I -  B^*).  
\end{eqnarray*}
By the definition of the inner product, followed by the triangle inequality,  and since  $(I-B^*)B^* = 0$, we further have: 
\begin{eqnarray}\nonumber
|\langle  \widetilde\Gamma- \Gamma , B^* - B \rangle| & \leq& 3 |B^*(\widetilde\Gamma-\Gamma)|_{\infty}|B^*(B^* - B)|_1 + |\langle (I - B^*)(\widetilde\Gamma-\Gamma)(I -  B^*), B^* - B\rangle| \\
&=&  3 |B^*(\widetilde\Gamma-\Gamma)|_{\infty}|B^*(B^* - B)|_1 +   |\langle\widetilde\Gamma-\Gamma, (I-B^*)B(I-B^*)\rangle|.  \label{unu}  
\end{eqnarray} 
By the duality of  the nuclear  $\| \ \|_*$  and operator $\| \ \|_{op}$  norms, we  have
\begin{eqnarray*}
  |\langle\widetilde\Gamma-\Gamma, (I-B^*)B(I-B^*)\rangle| 
	&\leq&  \|\widetilde\Gamma-\Gamma\|_{op}  \|(I-B^*)B(I-B^*)\|_*.
\end{eqnarray*}
We begin by bounding the  nuclear norm $\|(I-B^*)B(I-B^*)\|_*$. Since $(I-B^*)B(I-B^*) \in \mathcal S^+$, we have 
\[\|(I-B^*)B(I-B^*)\|_* = \tr((I-B^*)B(I-B^*)) =  \langle I-B^*,B(I-B^*) \rangle= \langle I-B^*,B \rangle.\] 
Using the fact that the sum of each row of $B$ is 1,
 we have
\begin{align}
\|(I-B^*)B(I-B^*)\|_*=	\langle I-B^*,B \rangle &= \tr(B) - \sum_{k=1}^{K}\sum_{a,b\in G_k}\frac{B_{ab}}{|G_k|}\nonumber \\
&= \tr(B)-K + \sum_{k\neq j}\sum_{a\in G_k,\ b\in G_j}\frac{B_{ab}}{|G_k|}  \nonumber\\
	& \leq \tr(B)-K +  \frac{1}{m}\sum_{k\neq j}|B_{G_jG_k}|_1 \, . \label{eq:trace-K}
\end{align}
 Next, we simplify the expression of  $|B^*(B^* - B)|_1=|B^*(I-B)|_{1}$.  
\beqn
|B^*(I - B)|_1&=&\sum_{j\neq k}\sum_{a\in G_j,\ b\in G_k}|(B^*B)_{ab} |+  \sum_{k=1}^K\sum_{a,b\in G_k}|[B^*(I - B)]_{ab}|\\
 &= &\sum_{j\neq k}\sum_{a\in G_j,\ b\in G_k} \frac{1}{|G_j|}\sum_{c\in G_j}B_{cb}+ \sum_{k=1}^K\sum_{a,b\in G_k}\frac{1}{|G_k|}\Big|1 -\sum_{c\in G_k}B_{cb}\Big|\\
 & = &  2 \sum_{j\neq k }|B_{G_jG_k}|_1\ , 
\eeqn
where we used again $B1=1$ and that the entries of $B$ are nonnegative.
Gathering the above bounds together with  (\ref{unu}) yields: 
\begin{equation}\label{start1}
|\langle \widetilde\Gamma- \Gamma, B^*-B \rangle| \leq   2\Bigg[\sum_{j\neq k }|B_{G_jG_k}|_1\Bigg]\pa{ {\| \widetilde\Gamma-\Gamma\|_{op}\over 2m}+ 3|B^*(\widetilde\Gamma-\Gamma)|_{\infty}} + \big[\tr(B)-K\big]\| \widetilde\Gamma-\Gamma\|_{op}.
\end{equation}
We bound below the two terms in the parenthesis of (\ref{start1}).  Since $n \Gamma^{-1/2}\widetilde\Gamma\Gamma^{-1/2}$ follows a Wishart distribution with $(n,p)$ parameters, we obtain by~\cite{Davidson2001} that
\begin{eqnarray}\label{eq:hatV_1}
\| \widetilde\Gamma-\Gamma\|_{op}\leq 4\|\Gamma\|_{op}\left[\sqrt{\frac{p}{n}}+ \frac{p}{n}\right]=4|\Gamma|_{\infty}\left[\sqrt{\frac{p}{n}}+  \frac{p}{n}\right]\ , 
\end{eqnarray}
with probability larger than $1-1/p$. 
%Following Koltchinski and Lounici, 2016 (see also Bunea and Xiao, 2015), we have 
% \begin{equation} \label{doi} \| \widetilde\Gamma- \Gamma\|_{op} \lesssim    \sqrt {\frac{r_e(\Gamma)}{n}} \leq   \sqrt {\frac{p}{n}} =  \sqrt {\frac{mK}{n}}, \end{equation}
%recalling that $|G_k| = m = \frac{p}{K}$. 
We now turn to $|B^*(\widetilde\Gamma-\Gamma)|_{\infty}$. For any $a,b$ in $[p]$, let $k$ be such that $a\in G_k$. We have
 \[
[B^*(\widetilde\Gamma-\Gamma)]_{ab} = \frac{1}{|G_k|} \sum_{l\in G_k} \left(\widetilde\Gamma _{lb}- \Gamma_{lb}\right) = \frac{1}{n|G_k|}  \sum_{l\in G_{k}}\sum_{i=1}^{n} [ \epsilon_{li}\epsilon_{bi} - \E(\epsilon_l\epsilon_b)] \ .
 \]
The sum $\sum_{l \sim a}\sum_{i=1}^{n} [ \epsilon_{li}\epsilon_{bi} - \E(\epsilon_l\epsilon_b)]$ is a centered quadratic form of $n(|G_k|+1)$ (or $n|G_k|$ if $b\in G_k$) independent normal  variables whose variances belong to $\{\gamma_l,\ l=b\text{ or }l\sim a\}$. Applying Lemma \ref{lem:quadratic} together with the assumption $\log(p)\leq c_{1}n$, we derive that, with probability larger than $1-c'/p$, 
\[
 \sum_{l \sim a}\sum_{i=1}^{n} [ \epsilon_{li}\epsilon_{bi} - \E(\epsilon_l\epsilon_b)]\leq c |\Gamma|_{\infty}\left(\sqrt{n|G_k|\log(p)}+ \log(p)\right)   \leq c' |\Gamma|_{\infty}\sqrt{n|G_k|\log(p)}
\]
simultaneously for all $l$ and all $b$. In the second inequality, we used the assumption $\log(p)\leq c_1 n$.
This yields 
\begin{equation} \label{eq:hatV_2} 
\mathbb{P}\left[|B^*(\widetilde\Gamma-\Gamma)|_{\infty} \leq c''|\Gamma|_{\infty} \sqrt{\frac{\log p}{mn}} \right]\geq 1-c'/p\ .
\end{equation}
Plugging into (\ref{start1}) the bounds derived in (\ref{eq:hatV_1}) and (\ref{eq:hatV_2}) above, and noticing from \eqref{eq:trace-K} that 
\beq\label{eq:trace-K2}
{1\over m} \sum_{j\neq k }|B_{G_jG_k}|_1+\tr(B)-K\geq 0,
\eeq
we obtain that 
\[
|\langle \widetilde\Gamma-\Gamma , B^*-B \rangle|   \leq c |\Gamma|_{\infty}\left( \sqrt{  \frac{\log p}{mn}   }  \vee    \sqrt {\frac{p}{m^2n}} \vee \frac{p}{nm}\right ) \sum_{j\neq k }|B_{G_jG_k}|_1+ 4|\Gamma|_{\infty}\left[\sqrt{\frac{p}{n}}+ \frac{p}{n}\right](\tr(B)-K)\ ,
\]
with  probability larger  than $ 1 - c'/p$.\\

\smallskip

\noindent {\it (b) Control of $\langle  \Gamma- \widehat\Gamma , B^* - B \rangle$.}\\*

\noindent We follow the same approach as for $\widetilde\Gamma-\Gamma$. 
The additional ingredient is that $\langle \Gamma- \widehat\Gamma, B^*-B \rangle=\langle \Gamma- \widehat\Gamma-\alpha I_{p}, B^*-B \rangle+ \alpha[K-\tr(B)]$ for any $\alpha\in\bbR$, since $\tr(B^*)=K$. 
Analogously to \eqref{start1}, for any $\alpha\in\bbR$,  the following holds:
\begin{align*}
|\langle \Gamma- \widehat\Gamma, B^*-B \rangle| &\leq|\langle \Gamma- \widehat\Gamma-\alpha I_{p}, B^*-B \rangle| + |\alpha [K-\tr(B)]| \\
&\leq   2\Bigg[\sum_{j\neq k }|B_{G_jG_k}|_1\Bigg]\pa{ {\| \Gamma-\widehat\Gamma-\alpha I_{p}\|_{op}\over 2m}+ 3|B^*(\Gamma-\widehat\Gamma)-\alpha I_{p}|_{\infty}}\\ &\quad + |\alpha [K-\tr(B)]| +\big[\tr(B)-K\big]\| \widetilde\Gamma-\Gamma-\alpha I\|_{op}\ .
\end{align*}
We fix $\alpha= | \Gamma-\widehat\Gamma|_{V}/2$ so that $| \Gamma-\widehat\Gamma-\alpha I_{p}|_{\infty}=| \Gamma-\widehat\Gamma|_{V}/2$.
Since $\alpha I_{p}$, $\Gamma$ and $\widehat\Gamma$ are diagonal matrices, the above inequality simplifies to 
\[
|\langle \Gamma- \widehat\Gamma, B^*-B \rangle| \leq   \frac{ 7|\Gamma-\widehat \Gamma|_{V}}{2m}\sum_{j\neq k }|B_{G_jG_k}|_1+ [\tr(B)-K]_{+}| \Gamma-\widehat\Gamma|_{V} \ .
\]
The proof of Lemma \ref{lem:covariance} is complete.
\end{proof}

\subsection{Poof of Corollary \ref{cor:consistency_2_steps2}}
At step 3 of PECOK, we have chosen a clustering algorithm which returns the partition $G^*$ when applied to the true partnership matrix $B^*$. Hence, PECOK returns $G^*$ as soon as $\widehat B=B^*$. The Corollary~\ref{cor:consistency_2_steps2} then follows by combining Theorem \ref{thm:consistency} and Proposition \ref{prp:control_u_gamma2}.

\subsection{Proof of Proposition \ref{prp:control_u_gamma2}}
Let $k$, $l_1$ and $l_2$ be such that $a\in G_k$ and $ne_1(a)\in G_{l_1}$ and $ne_2(a)\in G_{l_2}$.
 Starting from the identity $\bX_{:a}=\bZ_{:k}+\bE_{:a}$, we developp $\widehat\Gamma_{aa}$
\beqn
  \widehat\Gamma_{aa}&=& \frac{|\bE_{:a}|_2^2}{n} + \frac{1}{n} \langle \bZ_{:k}- \bZ_{:l_1}, \bZ_{:k}- \bZ_{:l_2}\rangle\\
 & &+ \frac{1}{n}\left[ \langle \bZ_{:k}- \bZ_{:l_1}, \bE_{:a}-\bE_{:ne_2(a)}\rangle +  \langle \bZ_{:k}- \bZ_{:l_2}, \bE_{:a}-\bE_{:ne_1(a)}\rangle \right]\\
 && + \frac{1}{n}\left[ \langle \bE_{:ne_1(a)}, \bE_{:ne_2(a)}\rangle - \langle \bE_{:a}, \bE_{:ne_1(a)}+\bE_{:ne_2(a)}\rangle\right]
\eeqn
Since $2xy\leq x^2+y^2$, the above expression decomposes as 
\begin{eqnarray}\label{eq:gamma_aa}
 \big|\widehat\Gamma_{aa}- \Gamma_{aa}\big|&\leq & \big|\frac{|\bE_{:a}|_2^2}{n}- \Gamma_{aa} \big|+ U_1 + 2U_2+3U_3\\
 U_1 &:=& \frac{1}{n} |\bZ_{:k}- \bZ_{:l_1}|_2^2+ \frac{1}{n}|\bZ_{:k}- \bZ_{:l_2}|_2^2\nonumber\\
 U_2&:= & \frac{1}{n}\sup_{k,j\in[K]}\sup_{b\in[p]} \langle \frac{\bZ_{:k}- \bZ_{:j}}{|\bZ_{:k}- \bZ_{:j}|_2}, \bE_{:b}\rangle^2 \ , \quad 
 U_3 :=  \frac{1}{n}\sup_{b\neq c }\langle \bE_{:b},\bE_{:c}\rangle \ .\nonumber
\end{eqnarray}
Recall that  all the columns of $\bE$ are independent and that $\bE$ is independent from $\bZ$. 
The terms $|E_{:a}|_2^2/n- \Gamma_{aa}$, $U_2$ and $U_3$ in \eqref{eq:gamma_aa} are quite straightforward to control as they either involve quadratic functions of Gaussian variables, suprema of Gaussian variables or suprema of centered quadratic functions of Gaussian variables. Applying the Gaussian tail bound and Lemma \ref{lem:quadratic} together with an union bound, and $\log(p)\leq c_{1} n$, we obtain
\beq\label{eq:upper_terme_simple}
\Big|\frac{|E_{:a}|_2^2}{n}- \Gamma_{aa} \big|+2U_2+3U_3\leq c|\Gamma|_{\infty}\sqrt{\frac{\log(p)}{n}} , %\\
%U_2\leq c|\Gamma|_{\infty}\frac{\log(p)}{n}\ ,\quad  U_3\leq c|\Gamma|_{\infty}\left[\sqrt{\frac{\log(p)}{n}}+ \frac{\log(p)}{n}\right]
\eeq
with probability higher than $1-1/p^2$. The main hurdle in this proof is to control the bias terms $|\bZ_{:k}- \bZ_{:l_i}|_2^2$ for $i=1,2$.

Since $m\geq 3$, there exists two indices $b_1$ and $b_2$ other than $a$ belonging to the  group $G_{k}$. As a consequence, $\bX_{:a}-\bX_{:b_i}= \bE_{:a}-\bE_{:b_{i}}$ is independent from $\bZ$ and from all the other columns of $\bE$. Hence, $\langle \bX_{:a}-\bX_{:b_{i}}, \tfrac{\bX_{:c}-\bX_{:d}}{|\bX_{:c}-\bX_{:d}|_2}\rangle$ is normally distributed with variance $2\Gamma_{aa}$ and it follows that, with probability larger than $1-p^{-2}$, 
\[V(a,b_1)\vee V(a,b_2) \leq c|\Gamma|^{1/2}_{\infty}\sqrt{\log(p)}\ .\]
The definition of $ne_1(a)$ and $ne_2(a)$ enforces that $V(a,ne_1(a))$ and $V(a,ne_2(a))$ satisfy the same bound.

When $k=l_{1}$ then $|\bZ_{:k}- \bZ_{:l_1}|_2^2=0$, so we only need to consider the case where $k\neq l_{1}$. 
Let $c\in G_k\setminus\{a\}$ and $d\in G_{l_{1}}\setminus\{ne_1(a)\}$, which exists since $m\geq 3$. The above inequality for $V(a,ne_1(a))$ implies 
\beq\label{eq:condition_Xa}
\big|\langle \bX_{:a}-\bX_{:ne_1(a)},\bX_{:c}-\bX_{:d}\rangle\big| \leq c|\Gamma|^{1/2}_{\infty}\sqrt{\log(p)}|\bX_{:c}-\bX_{:d}|_2.
\eeq
This inequality is the key to control the norm of $t= \bZ_{:k}-\bZ_{:l_1}$. 
Actually, since $a,c\in G_{k}$ and $ne_{1}(a),d\in G_{l_{1}}$, 
we have
\begin{align*}
 \big|\langle \bX_{:a}-\bX_{:ne_1(a)},&\bX_{:c}-\bX_{:d}\rangle\big| 
 = \big| |t|_{2}^2+\langle t,E_{a}-E_{ne_{1}(a)}+E_{c}-E_{d}\rangle +\langle \bE_{:a}-\bE_{:ne_1(a)},\bE_{:c}-\bE_{:d}\rangle\big|\\
&\geq \frac{|t|_2^2}{2}-\frac{1}{2}\big| \langle \frac{t}{|t|_2}, \bE_{:a}-\bE_{:ne_1(a)}+\bE_{:c}- \bE_{:d}\rangle\big|^2 - \big|\langle \bE_{:a}-\bE_{:ne_1(a)},\bE_{:c}-\bE_{:d}\rangle\big|.   
\end{align*}
Applying again a Gaussian deviation inequality and Lemma \ref{lem:quadratic} simultaneously for all $a,b,c,d\in [p]$ and $k,l\in [K]$, we derive that with probability larger than $1-p^{-2}$,
\[
\big|\langle \bX_{:a}-\bX_{:ne_1(a)},\bX_{:c}-\bX_{:d}\rangle\big|\geq {1\over 2}|t|_2^2 - c|\Gamma|_{\infty}\sqrt{n\log(p)}\ ,
\]
since $\log(p)\leq c_{1} n$.
Turning to the rhs of \eqref{eq:condition_Xa}, we have $|\bX_{:c}-\bX_{:d}|_2\leq |t| + |\bE_{:c}-\bE_{:d}|_2$. Taking an union bound over all possible $c$ and $d$, we have $|\bE_{:c}-\bE_{:d}|_2\leq c|\Gamma|_{\infty}^{1/2}\sqrt{\log(p)}\leq c'|\Gamma|_{\infty}^{1/2}n^{1/2}$ with probability larger than $1-p^{-2}$. Plugging these results in \eqref{eq:condition_Xa}, we arrive at 
\[
 |t|_2^2 - c_1 |t|_2|\Gamma|_{\infty}^{1/2}\sqrt{\log(p)}\leq c_2 |\Gamma|_{\infty}\sqrt{n\log(p)}\ .
\]
This last inequality together with $\log(p)\leq c_{1} n$ enforce that 
\[
|t|_{2}^2= |\bZ_{:k}-\bZ_{:l_1}|_2^2\leq c |\Gamma|_{\infty}\sqrt{n\log(p)}.
\]
Analogously, the same bound holds for  $|\bZ_{:k}-\bZ_{:l_2}|_2^2$.
Together with \eqref{eq:gamma_aa} and \eqref{eq:upper_terme_simple}, we have proved that 
\[\big|\widehat\Gamma_{aa}- \Gamma_{aa}\big|\leq c  |\Gamma|_{\infty}\sqrt{\frac{\log(p)}{n}}\ ,\]
with probability larger than $1-p^{-2}$. The result follows.

\subsection{Proof of Theorem \ref{thm:Kselection}}
We shall follow the same approach as in the proof of Theorem \ref{thm:consistency}. We need to prove that 
 \[ \langle \widehat{\Sigma}, B^*-B \rangle + \widehat{\kappa} [\tr(B)-K] > 0, \  \mbox{ for all} \  B \in \mathcal{C}_0\setminus\{B^*\}  .\]
 As in that previous proof, we introduce the matrix $W$, so that it suffices to prove that 
\beq\label{eq:main_objective_bis}
R(B):=\langle W, B^*-B \rangle + n\widehat{\kappa} [\tr(B)-K] > 0,  \  \mbox{ for all} \  B \in \mathcal{C}_0\setminus\{B^*\}\ .
\eeq
We use the same decomposition as previously.  Applying \eqref{eq:main1} together with Lemmas \ref{lem:W2} and \ref{lem:covariance}, we derive that, with probability larger than $1-c/p$, 
 \begin{eqnarray}\label{main-K}
 \lefteqn{\langle W, B^*-B \rangle \geq  \sum_{j\neq k} \Bigg[{1\over 2}|\bZ_{:j}-\bZ_{:k}|^2_2  -  c_1 \sqrt{\log(p)} |\Gamma|_{\infty}^{1/2}|\bZ_{:j}-\bZ_{:k}|_2} \\&& -c_2 \frac{n|\widehat\Gamma-\Gamma|_{V}}{m}  - c_{2}|\Gamma|_{\infty}\left( \sqrt{  \frac{n\log p}{m}   }  \vee    \sqrt {\frac{np}{m^2}} \vee \frac{p}{m}\right )\Bigg] |B_{G_jG_k}|_1\ , \nonumber \\
 && - 4n|\Gamma|_{\infty}\left(\sqrt{\frac{p}{n}}+\frac{p}{n}\right)(\tr(B)-K)-  n[\tr(B)-K]_{+}| \Gamma-\widehat\Gamma|_{V} \ .\nonumber
\end{eqnarray}
As in \eqref{eq:control_distance}, we use that with high probability  $|\bZ_{:j}-\bZ_{:k}|^2_2$ is larger than $n\Delta(C)/2$.  Condition \eqref{eq:assumption2-bis} then enforces, that with high probability, the term inside
the square brackets in \eqref{main-K} is larger than $n\Delta(C)/8$. As a consequence, we have
\begin{eqnarray}
 R(B)\geq  {n\over 8}\Delta(C) \sum_{j\neq k} |B_{G_jG_k}|_1 
  +  n(\tr(B)-K)\left[ \widehat{\kappa} - 4|\Gamma|_{\infty}\left(\sqrt{\frac{p}{n}}+ \frac{p}{n}\right)  - \mathbf{1}_{\tr(B)\geq K}| \Gamma-\widehat\Gamma|_{V}\right] \ , \label{eq:main_adaptation}
\end{eqnarray}
uniformly over all $B\in \cC_0$. To finish the proof, we divide the analysis intro three cases depending on the values of $\tr(B)$.\smallskip

\noindent 1. If $\tr(B)>K$, we apply Condition \eqref{eq:condition_widehat_kappa} to get $\widehat{\kappa} - 4|\Gamma|_{\infty}\left(\sqrt{\frac{p}{n}}+ \frac{p}{n}\right)  - | \Gamma-\widehat\Gamma|_{V}>0$, which implies that the right-hand side in \eqref{eq:main_adaptation} is positive. 
\smallskip

\noindent 2. If $\tr(B)=K$, the right-hand side in \eqref{eq:main_adaptation} is positive except if $\sum_{j\neq k} |B_{G_jG_k}|_1=0$, which implies $B=B^*$ by Lemma \ref{lem:support}. 
\smallskip

\noindent 3. Turning to the case $\tr(B)<K$, \eqref{eq:main_adaptation} implies that 
\begin{eqnarray}
 R(B) \geq {n\over 8}\Delta(C) \sum_{j\neq k} |B_{G_jG_k}|_1 - \widehat{\kappa} n(K-\tr(B)) \ . \label{eq:main_adaptation2bis}
\end{eqnarray}
It turns out, that when $\tr(B)<K$, the support of $B$ cannot be included in the one of $B^*$ so that $ \sum_{j\neq k} |B_{G_jG_k}|_1$ is positive. 
%Let us derive a lower bound of this quantity. Since $B^*$ is a projection matrix, it follows that $(I-B^*)B(I-B^*)\in \mathcal{S}^+$ and $tr((I-B^*)B(I-B^*))= tr((I-B^*)B)$ is therefore nonnegative. Working out this trace, we have
%\[
% tr((I-B^*)B)= \tr(B) - \sum_{k=1}^K\frac{|B_{G_kG_k}|_1}{|G_k|}= \tr(B)-K + \sum_{k=1}^K\frac{\sum_{j\neq k} |B_{G_{k}G_{j}}|_1}{|G_k|}\leq \tr(B)-K + \frac{ \sum_{j\neq k }|B_{G_jG_k}|_1 }{m}\ ,
%\]
%As a consequence, 
Actually, \eqref{eq:trace-K2} ensures that
$\sum_{j\neq k }|B_{G_jG_k}|_1 \geq [K-\tr(B)]m$, so together with \eqref{eq:main_adaptation2bis}, this gives us
 \[
 R(B) \geq  n[K-\tr(B)] \left[ \frac{m}{8}\Delta(C) -  \widehat{\kappa} \right]\ ,
\]
which is positive by condition \eqref{eq:assumption2-bis}.

\subsection{Proof of Theorem \ref{prp:minimax_lower_bound}}

The proof is based on a careful application of Fano's lemma. Before doing this, we need to reformulate the clustering objective into a discrete estimation problem.

\paragraph{Construction of the covariance matrices $\Sigma^{(j)}$.}

Let $A^{(0)}$ be the assignment matrix  such that the $m$ first variables belong to the first group, the next $m$ belong to the second group and so on. In other words, 
\[A^{(0)} = \begin{bmatrix}
1&0&0&\hdots&&0\\
\vdots\\
1&0&0&\hdots&&0\\
0&1&0&\hdots&&0\\
&\vdots&\ddots\\
&&&&0&1\\
&&&&&\vdots\\
&&&&&1\\
\end{bmatrix}
\text{ so that }
\Sigma^{(0)} = \begin{bmatrix}
1+\tau&\tau&\tau&0&\hdots\\
\tau & \ddots & \tau\\
\tau&\tau&1+\tau\\

0&&&1+\tau&\tau&\tau&0&\hdots\\
\vdots&&&\tau&\ddots&\tau\\
&&&\tau&\tau&1+\tau\\
&&&0&&&\ddots\\
&&&\vdots&&&&\\

\end{bmatrix}\ , \]
where  $\Sigma^{(0)}= A^{(0)} \tau I_K {A^{(0)}}^t + I_p$. Note that the associated partition for $G^{(0)}$ is $\{ \{1...m\}...\{p-m+1...p\}\}$. For any $a=m+1,\ldots, p$, denote $\mu_{a}$ the transposition between $1$ and $a$ in $\{1...p\}$. Then, for any $a=m+1,\ldots,p$, define the 
assignement matrix $A^{(a)}$ and $\Sigma^{(a)}$ by 
\[A^{(a)}_{ij}= A^{(0)}_{\mu_{a}(i),j}\ ,\quad   \Sigma^{(a)}_{ij}= \Sigma^{(0)}_{\mu_{a}(i),\mu_a(j)}\ .\]
In other words, the corresponding partition $G^{(a)}$ is obtained from $G^{(0)}$  by exchanging the role of the first and the $a$-th node. 
\medskip

Also define the subset $M:=\{0, m+1, m+2,\ldots, p\}$.
Equipped with these notations, we observe that the minimax error probability of perfect recovery is lower bounded by 
$\overline{\mathbf{R}}^*[\tau,n,m,p]\geq \inf_{\hat{G}} \max_{j\in M } \P_{\Sigma^{(j)}} \big( \hat{G} \neq G_j \big)$.
According to Birg\'e's version of Fano's Lemma (see e.g. \cite[Corollary 2.18]{massart}), 
\[
 \inf_{\hat{G}} \max_{j\in M } \P_j \big( \hat{G} \neq G_j \big)\geq  \frac{1}{2e+1}\bigwedge \left(1-\frac{\sum_{j\in M\setminus\{0\}}\KL(\P^{\otimes n}_{\Sigma^{(j)}},\P^{\otimes n}_{\Sigma^{(0)}})}{(|M|-1)\log(|M|)}\right).
\]
By symmetry, all the Kullback divergences are equal. Since $2e/(2e+1)\geq 0.8$ and $1/(2e+1)\geq 1/7$, we arrive at
\beq\label{eq:fano} 
 \overline{\mathbf{R}}^*[\tau,n,m,p] \geq  1/7\ , \quad  \text{ if }\quad  n\KL(\P_{\Sigma^{(m+1)}},\P_{\Sigma^{(0)}})\leq 0.8 \log(p-m+1)\ .
\eeq
As the derivation of the Kullback-Leibler discrepancy is involved, we state it here and postpone its proof  to the end of the section.

\begin{lem}\label{lem:computation_kullback_distance}
For any $\tau>1$ and any integers $p$ and $m$, we have
\beq \label{eq:Kullback} 
\KL(\P_{\Sigma^{(m+1)}},\P_{\Sigma^{(0)}})= \frac{2 (m-1)\tau^2}{1+m\tau}
\eeq
\end{lem}
As a consequence, the minimax error probability of perfect recovery $\overline{\mathbf{R}}^*[\tau,n,m,p]$ is larger than $1/7$  as soon as 
\[
 \frac{2n (m-1)\tau^2}{1+m\tau}\leq 0.8 \log(p-m+1)\ . 
\]
This last condition is satisfied as soon as 
\[
 \tau \leq c \left[\sqrt{\frac{\log(p)}{n(m-1)}}\bigvee \frac{\log(p)}{n} \right]\ ,
\]
for some numerical constant $c>0$.

\subsection{Proof of Lemma \ref{lem:computation_kullback_distance}}
The Kullback-Leibler divergence between two centered normal distributions writes as 
\beq\label{eq:formula_kullback}
\KL(\P_{\Sigma^{(m+1)}}, \P_{\Sigma^{(0)}}) = \frac{1}{2} \big[ - \log \mathrm{det} \big((\Sigma^{(0)})^{-1}\Sigma^{(m+1)}\big) + \text{trace} \big((\Sigma^{(0)})^{-1}\Sigma^{(m+1)} - I_p\big)
\big] \ , 
\eeq
so that we only have to compute the determinant and the trace of $A:= (\Sigma^{(1)})^{-1}\Sigma^{(m+1)}$. We shall see that see that $A$ is a rank 2 perturbation of the identity matrix, so that we will only need to compute its two eigenvalues different from zero.

Observe that for $i=0,m+1$, the matrices $A^{(i)}{A^{(i)}}^t$ admit exactly $K$ non-zero eigenvalues that are all equal to $m$. As a consequence, we can decompose $A^{(i)}{A^{(i)}}^t = m \sum_{k=1}^{K}u_k^{(i)}(u_k^{(i)})^t$ where $u_k^{(i)}$ is a unit vector whose non zero components are all equal to $1/m$ and correspond to the $k$-th group in $G^{(i)}$. Note that $u_k^{(0)}=u_{k}^{(m+1)}$ for $k=3,\ldots, K$ as $A^{(0)}$ and $A^{(m+1)}$ only differ by rows $1$ and $m+1$. The orthogonal projector $P_i= \sum_{k=1}^{K}u_k^{(i)}(u_k^{(i)})^t$ satisfies 
\[\Sigma^{(i)}= m \tau P_i + I_p = (1+m \tau)P_i + (I_p-P_i)\ .\]
Since $P_i$ and $I_p-P_i$ are orthogonal,
\[(\Sigma^{(i)})^{-1} = (1+m\tau)^{-1}P_i + (I_p-P_i)= I_p-\frac{m\tau}{1+m\tau}P_i\]
As a consequence of the above observations, we have
\beqn
A &=& I_p + (\Sigma^{(0)})^{-1}\Big[\Sigma^{(m+1)}-\Sigma^{(0)}\Big]\\
& = & I_p + m\tau\big(P_{m+1}- P_{0}\big) -  \frac{m^2\tau^2}{1+m\tau}P_{0}(P_{m+1}-P_{0}) =: I_p + B
\eeqn 
The matrices $P_{0}$ and $P_{m+1}$ are $k-1$ block diagonal with a first block of size $2m\times 2m$. Besides, $P_{0}$ and $P_{m+1}$ take the same values on all the $K-2$ remaining blocks. To compute the non-zero eigenvalues of $B$, we only to consider the restrictions $\overline{P}_{0}$ and $\overline{P}_{m+1}$  of $P_{0}$ and $P_{m+1}$ to the first $2m\times 2m$ entries. Also observe that the matrices $\overline{P}_{0}$ and $\overline{P}_{m+1}$ are $4\times 4$ block-constant, with block size 
\[
{\scriptsize
\begin{bmatrix}
1\times 1&1\times (m-1) &1\times 1& 1\times (m-1)\\
 (m-1) \times1 & (m-1) \times (m-1) & (m-1) \times1 & (m-1) \times (m-1) \\
1\times 1&1\times (m-1) &1\times 1& 1\times (m-1)\\
(m-1) \times1 & (m-1) \times (m-1) &(m-1) \times1&(m-1) \times (m-1)
\end{bmatrix}}\]
and the entries are 
\[
{\scriptsize
m\overline{P}_{0} = \begin{bmatrix}
1&1&0&0\\
1& 1 & 0 & 0\\
0&0&1&1\\
0&0&1&1
\end{bmatrix}
\text{ and }
m \overline{P}_{m+1} = \begin{bmatrix}
1&0&0&1\\
0& 1 & 1 & 0\\
0&1&1&0\\
1&0&0&1 \ .
\end{bmatrix}
}
\]
As a consequence, the non zero eigenvalues of $B$ are the same as those of \[
C:=  m\tau \big(\underline{P}_{m+1}- \underline{P}_{0}\big) -  \frac{m^2\tau^2}{1+m\tau}\underline{P}_{0}(\underline{P}_{m+1}-\underline{P}_{0})                                                                           \]
 where $\underline{P}_{m+1}$ and $\underline{P}_{0}$ are two $4\times 4$ matrices
\[
{\scriptsize
m\underline{P}_{0} = \begin{bmatrix}
1&(m-1)&0&0\\
1& (m-1) & 0 & 0\\
0&0&1&(m-1)\\
0&0&1&(m-1)
\end{bmatrix}
\text{ and }
m \underline{P}_{m+1} = \begin{bmatrix}
1&0&0&(m-1)\\
0& (m-1) & 1 & 0\\
0&(m-1)&1&0\\
1&0&0&(m-1) \ .
\end{bmatrix}
}
\]
Working out the product of matrices, we get
\[
{\scriptsize
 C = - \tau \begin{bmatrix}
0&(m-1)&0&-(m-1)\\
1& 0 & -1 & 0\\
0&-(m-1)&0&(m-1)\\
-1& 0 & 1 & 0\\
\end{bmatrix}  + \frac{(m-1)\tau^2}{1+m\tau}
\begin{bmatrix}
1&1&-1&-1\\
1& 1 & -1 & -1\\
-1&-1&1& 1\\
-1&-1&1&1 
\end{bmatrix}
}
\]
We observe  that these two matrices have their first (resp. second) and third (resp. fourth) lines and columns opposite to each other. As a consequence, the two non-zero eigenvalues of $C$ are the same as those of 
\beqn 
{\scriptsize
 D}&: =& {\scriptsize -2\tau \begin{bmatrix}
0&(m-1)\\
1& 0 
\end{bmatrix}  + \frac{2(m-1)\tau^2}{1+m\tau}
\begin{bmatrix}
1&1\\
1&1
\end{bmatrix} }\\
&=& {\scriptsize \frac{2\tau}{1+m\tau}\begin{bmatrix}
(m-1)\tau & -(m-1)[1+(m-1)\tau]\\
-(1+\tau)& (m-1)\tau
\end{bmatrix}}
\ .
\eeqn
Straightforward computations then lead to 
\[tr(D) = \frac{4(m-1)\tau^2}{1+m\tau} \ , \quad \mathrm{det}(D)= -tr(D)  \]

Coming back to \eqref{eq:formula_kullback}, we have
\beqn 
2\KL(\P_{\Sigma^{(m+1)}}, \P_{\Sigma^{(0)}}) &=&  - \log \mathrm{det} \big(A \big) + \text{trace} \big(A- I_p\big)
\big]\\&=& - \log \mathrm{det}(I+D)+ tr(D)\\
& =&  tr(D) - \log\big[1+tr(D)+ det(D)\big]\\
& =&\frac{4(m-1)\tau^2}{1+m\tau}  \  .
\eeqn

\subsection{Proof of Theorem \ref{prp:spectral_clustering}}
The proof is based on the following Lemma by Lei and Rinaldo \cite{LeiRinaldo}.
 \begin{lem}\label{lem:kmeans}
Let $M$ be any matrix of the form $M=AQ$ where $A\in \cA_{p,K}$ is a membership matrix and $Q\in \mathbb{R}^{K\times q}$, and  denote by $\delta$ the minimal distance between two rows of $Q$. Then, there exists a constant $c_{\eta}$,  such that, for any matrix $M'$ fulfilling $\|M-M'\|_{F}^2< m\delta^2/c_{\eta}$, the classification of the rows of $M'$ by an $\eta$-approximate $K$-means provides a clustering $\widehat G$ fulfilling 
$$\bar{L}(\hat G,G)\leq c_{\eta} {\|M-M'\|_{F}^2\over m\delta^2}.$$
\end{lem}

  We start with the following observation.
 Since $\|\Sigma\|_{op}\geq m\|C\|_{op}\geq m \lambda_K(C)$, Condition
 \eqref{eq:spectral} enforces that
  \beq\label{eq:condition_Re_log}
 \frac{Re(\Sigma)\vee \log(p)}{n}\leq 1/c^{2}_{\eta}\ .
 \eeq
 Let  $U$ be a $K\times p$ matrix which gathers the eigenvectors of $A CA^t$ associated to the $K$ leading eigenvalues.
The associated eigenvectors are block constant. Therefore $U_{0}=AQ_{0}$, and since $A^tA=mI$, the matrix $\sqrt{m}Q_{0}$ is orthogonal.

 We apply Lemma \ref{lem:kmeans} with
 $M'=\widehat U$ and $M=U_{0}\widehat O$, where $\widehat O$ is a  $K\times K$ orthogonal matrix to be chosen. 
We have $M=AQ$ with $\sqrt{m}Q=\sqrt{m}Q_{0}\widehat O$ orthogonal. In particular, the minimal distance between two rows of $Q$ is $\delta=\sqrt{2/m}$. Lemma \ref{lem:kmeans} ensures that
\begin{equation}\label{eq:spec1}
\bar{L}(\hat G_{S},G) \leq c_{\eta} {\|\widehat U-U_{0}\widehat O\|_F^2 \over 2},
\end{equation}
whenever the right-hand side is smaller than 1.
By Davis-Kahan inequality (e.g. \cite{LeiRinaldo}), there exists an orthogonal matrix $\widehat O$ such that
\begin{equation}\label{eq:spec2}
 \|\widehat U-U_{0}\widehat O\|_F^2\leq {8K  \|\widetilde \Sigma-ACA^t\|^2_{op} \over m^2 \lambda^2_{K}(C)}\ .
 \end{equation}
 We can upper-bound the operator norm of $\widetilde{\Sigma}-A CA^t$ by
$$\|\widetilde{\Sigma}-A CA^t\|_{op}\leq \|\widehat{\Sigma}-\Sigma\|_{op}+ \|\widehat{\Gamma}-\Gamma\|_{op}\ .$$
%For any symmetric matrix $S$, since the matrix $\|S\|_{op}I-S$ is positive semi-definite, it has a non-negative diagonale and hence $\|Diag(S)\|_{op}=|Diag(S)|_{\infty}\leq  \|S\|_{op}$. Therefore,  in the right-hand side, the second term can be upper-bounded by the first term. 
According to Theorem 1 in \cite{koltchinskii2014concentration} (see also \cite{BuneaEtal}), there exists a constant $c>0$ such that, with  probability at least $1-1/p$
\begin{eqnarray*}\nonumber
\|\widehat{\Sigma}-\Sigma\|_{op}&\leq& c \|\Sigma\|_{op} \left(\sqrt{Re(\Sigma)\over n}\bigvee {Re(\Sigma)\over n}\bigvee \sqrt{\log(p)\over n} \bigvee \frac{\log(p)}{n} \right)\\
&\leq&  c \|\Sigma\|_{op} \left(\sqrt{Re(\Sigma)\over n} \bigvee \sqrt{\log(p)\over n}  \right) ,
\end{eqnarray*}
where we used \eqref{eq:condition_Re_log} in the second line.

Then, using that $\|\widehat{\Gamma}-\Gamma\|_{op}=|\widehat{\Gamma}-\Gamma|_{\infty}$ and Proposition \ref{prp:control_u_gamma2} together with $|\Gamma|_{\infty}\leq \|\Sigma\|_{op}$, we obtain the inequality
\begin{equation}\label{eq:KarimVladimir}
 \|\widetilde{\Sigma}-A CA^t\|_{op} \leq c  \|\Sigma\|_{op}\left(\sqrt{Re(\Sigma)\over n} \bigvee \sqrt{\log(p)\over n}  \right)\ , 
\end{equation}
with probability at least $1-c/p$.
  So combining (\ref{eq:spec1}), with (\ref{eq:spec2}) and (\ref{eq:KarimVladimir})
 we obtain the existence of $c'_{\eta}>0$ such that we have
$$\bar{L}(\hat G_S,G) \leq {c'_{\eta}K \|\Sigma\|_{op}^2\over m^2\lambda_{K}(C)^2}\left(\sqrt{\frac{Re(\Sigma)}{n}}\bigvee  \sqrt{\log(p)\over n}\right)^2,$$
with probability at least $1-c/p$, whenever the right-hand side is smaller than 1.  
 The proof of Theorem \ref{prp:spectral_clustering} follows. 
 
\subsection{Proof of Lemma  \ref{CSC}}
%\begin{proof} 
We recall that $\widehat{U}$ is  the $p\times K$ matrix stacking the $K$ leading  eigenvectors of 
$\widetilde{\Sigma} = \widehat{\Sigma} - \widehat{\Gamma}$.
We first prove that the matrix $\widehat{U}\widehat{U}^t$ is solution of (\ref{SDP2}).

Let us write $\widetilde \Sigma=\widetilde U \widetilde D \widetilde U^t$ for a diagonalisation of 
$\widetilde \Sigma$ with $\widetilde U$ orthogonal and $\widetilde D_{11}\geq \ldots \geq  \widetilde D_{pp}\geq 0$.  We observe that $\langle \widetilde \Sigma,B\rangle= \langle \widetilde D,\widetilde U^t B \widetilde U\rangle$, and that $B\in \overline{\mathcal{C}}$ iff $\widetilde U^t B \widetilde U \in \overline{\mathcal{C}}$  since the matrix $\widetilde B=\widetilde U^t B \widetilde U$ has the same eigenvalues as $B$. We observe also that $\widehat{U}\widehat{U}^t=\widetilde U\Pi_{K}\widetilde U^t$, where $\Pi_{K}$ is the diagonal matrix, with 1 on the first $K$ diagonal elements and 0 on the $p-K$ remaining ones.
So proving that $\overline{B}=\widehat{U}\widehat{U}^t$ is solution of (\ref{SDP2})
is equivalent to proving that 
$$\Pi_{K}=\argmax_{\widetilde B\in \overline{\mathcal{C}}} \langle \widetilde D,\widetilde B\rangle.$$
 Let us prove this result.

To start with, we notice that 
$$ \sum_{k=1}^K\widetilde D_{kk}=\max_{0\leq \widetilde B_{kk}\leq 1;\ \sum_{k}\widetilde B_{kk}=K} \langle \widetilde D,\widetilde B\rangle.$$
Since the condition $I \succcurlyeq \widetilde B\succcurlyeq 0$ enforces $0\leq \widetilde B_{kk}\leq 1$, we  have $\overline{\mathcal{C}}\subset \{ B:  0\leq \widetilde B_{kk}\leq 1;\ \sum_{k}\widetilde B_{kk}=K\}$ and then
$$\max_{\widetilde B\in \overline{\mathcal{C}}} \langle \widetilde D,\widetilde B\rangle\leq \sum_{k=1}^K\widetilde D_{kk}=\langle \widetilde D,\Pi_{K}\rangle.$$
 Hence $\Pi_{K}$ is solution to the above maximisation problem and $\overline B=\widetilde U \Pi_{K} \widetilde U^t=\widehat U \widehat U^t$. 
 
To conclude the proof, we notice that $\widehat U_{a:}\widehat U^t$ is an orthogonal transformation of $\widehat U_{a:}$, so we obtain the same results when applying a rotationally invariant clustering algorithm to the rows of $\widehat U$ and to the rows of $\widehat U \widehat U^t$. 
%\end{proof} 

\bibliography{biblio}
\bibliographystyle{plain}

\appendix
\section*{Appendix}
\section{An alternative estimator of $\Gamma$}\label{app:gamma}
We propose here a more simple estimator of $\Gamma$. It has the nice feature to have a smaller computational complexity than \eqref{eq:estim:gamma2}, but the drawback to have fluctuations possibly proportional to $|\Sigma|_{\infty}^{1/2}$. 

For any $a\in [p]$, define 
\beq \label{eq:definition_ne(a)}
ne(a):= \argmin_{b\in [p]\setminus\{a\}}\,  \max_{c\neq a,b}\big|\langle \bX_{:a}-\bX_{:b}, \frac{\bX_{:c}}{|\bX_{:c}|_2}\rangle\big| \ ,
\eeq
the ``neighbor'' of $a$, that is the variable $\bX_{:b}$ such that the covariance $\langle \bX_{:b} ,\bX_{:c}\rangle$ is most similar to $\langle \bX_{:a} ,\bX_{:c}\rangle$, this for all variables $c$. It is expected that $ne(a)$ belongs to the same group of $a$, or if it is not the case that the difference $C^*_{kk}-C^*_{kj}$, where $a\in G^*_k$ and $ne(a)\in G^*_j$, is small. 

Then, the diagonal matrix $\widehat\Gamma$ is defined by 
\beq\label{eq:estim:gamma}
\widehat\Gamma_{aa}= {1\over n}\langle \bX_{:a}- \bX_{:ne(a)},\bX_{:a}\rangle,\ \textrm{for}\ a=1,\ldots,p.
\eeq 
In population version, this quantity is of order $\Gamma_{aa}+ C^*_{k(a)k(a)}-C^*_{k(a)k(ne(a))}$ ($k(a)$ and $k(ne(a))$ respectively stand for the group of $a$ and $ne(a)$)  and should therefore be of order $\Gamma_{aa}$ if the last intuition is true. As shown by the following proposition, the above discussion can be made rigorous.

 \begin{prp}\label{prp:control_u_gamma}
 There exist three numerical constants $c_1$--$c_3$  such that the following holds. Assume that $m\geq 2$ and that $\log(p)\leq c_1 n$. 
 With probability larger than $1-c_{3}/p$, the estimator $\widehat\Gamma$ defined by (\ref{eq:estim:gamma}) satisfies
 \beq\label{eq:control_check_gamma}
|\widehat\Gamma - \Gamma|_{V}\leq 2 |\widehat\Gamma - \Gamma|_{\infty}\leq c_2 |\Gamma|^{1/2}_\infty|\Sigma|^{1/2}_{\infty}\sqrt{\frac{\log(p)}{n}}\ .
 \eeq

 \end{prp}

The PECOK estimator  with $\widehat \Gamma$ defined by (\ref{eq:estim:gamma}) then fulfills the following recovering property.

\begin{cor}\label{thm:consistency_2_steps}
There exist $c_1,\ldots, c_3$ three positive constants such that the following holds. 
Assuming that   $\widehat \Gamma$ is defined by (\ref{eq:estim:gamma}), $\log(p)\leq c_1 n$, 
and that 
\begin{equation}\label{eq:assumption2_2steps}
 \Delta(C^*) \geq c_2 \left[|\Gamma|_{\infty}\left\{\sqrt{  \frac{\log p}{mn}   }+    \sqrt{\frac{p}{nm^2}} + \frac{\log(p)}{n}+ \frac{p}{nm}\right\} + |\Gamma|_{\infty}^{1/2}|C^*|^{1/2}_{\infty} \sqrt{\frac{\log(p)}{nm^2}} \right]\ ,
\end{equation}
then we have
$\widehat{B} = B^*$ and $\widehat G=G^*$, with probability higher than $1 - c_3/p$.
\end{cor}

The additional term $|\Gamma|_{\infty}^{1/2}|C^*|^{1/2}_{\infty}\sqrt{\frac{\log(p)}{nm^2}}$ term is smaller than $|\Gamma|_{\infty}\sqrt{\frac{\log(p)}{nm}}$ when $|C^*|_{\infty}\leq m |\Gamma|_{\infty}$, which is likely to occur when $m$ is large.  
%{\color{blue} Should we further discuss this result? }

\subsection{Proof of Proposition \ref{prp:control_u_gamma}}

 Consider any $a\in [p]$, let $k$ be the group such that $a\in G^*_k$. We now divide the analysis into two cases: (i) $ne(a)\in G^*_k$; (ii) $ne(a)\notin G^*_k$. 
 
 In case (i), we directly control the difference:
\beqn
 |\widehat\Gamma_{aa} - \Gamma_{aa}|&=&\big| \langle \bX_{:a}- \bX_{:ne(a)},\bX_{:a}\rangle/n - \Gamma_{aa}\big|\\
 &\leq & \big|\Gamma_{aa}- |\bE_{:a}|^2_2/n\big| + \big|\langle \bE_{:ne(a)}, \bE_{:a}\rangle/n\big|+ \big|\langleÂ \bE_{:a}-\bE_{:ne(a)} , \bZ_{:k} \rangle/n\big|\\
 &\leq &\big|\Gamma_{aa}- |\bE_{:a}|^2_2/n\big| + \big|\langle \bE_{:ne(a)}, \bE_{:a}\rangle/n\big|+ \big|\langleÂ \bE_{:a}-\bE_{:ne(a)} , \bZ_{:k} \rangle/n\big|\\
 &\leq & \big|\Gamma_{aa}- |\bE_{:a}|^2_2/n\big|+ \sup_{c\in G^*_k\setminus\{a\}}\left(\big|\langle \bE_{:c}, \bE_{:a}\rangle/n\big|+ \big|\langle\bE_{:a}-\bE_{:c} , \bZ_{:k} \rangle/n\big|\right)
\eeqn
The random variable $|\bE_{:a}|^2_2/\Gamma_{a,a}$ follows a $\chi^2$ distribution with $n$ degrees of freedom whereas the remaining variables are centered quadratic form of independent Gaussian variables. Applying the deviation inequality for Gaussian quadratic forms (\ref{eq:quadratic}) together with an union bound, we arrive at
\[
 |\widehat\Gamma_{aa} - \Gamma_{aa}|\leq c |\Gamma|^{1/2}_{\infty}|\Sigma|^{1/2}_{\infty} \sqrt{\frac{\log(p)}{n}} ,
\]
with probability larger than $1-1/p^{2}$. 
\medskip 

Let us turn to case (ii): $ne(a)\notin G^*_k$. Let $b\in G^*_k$, with $b\neq a$. We have
\beqn
 |\widehat\Gamma_{aa} - \Gamma_{aa}|&=&\big| \langle \bX_{:a}- \bX_{:ne(a)},\bX_{:a}\rangle/n - \Gamma_{aa}\big|\\
&\leq & \big| \langle \bX_{:a}- \bX_{:ne(a)},\bE_{:a}\rangle/n - \Gamma_{aa}\big| + \big| \langle \bX_{:a}- \bX_{:ne(a)},\bZ_{:k}\rangle/n \big|\\
&\leq & \big||\bE_{:a}|_{2}^2/n - \Gamma_{aa}\big|+  \big|\langle \bZ_{:a}- \bX_{:ne(a)},\bE_{:a}\rangle/n + \big| \langle \bX_{:a}- \bX_{:ne(a)},\bX_{:b}\rangle/n \big|\\
&&+  \big| \langle \bX_{:a}- \bX_{:ne(a)},\bE_{:b}\rangle/n \big|\ ,
\eeqn
where we used $\bX_{:b}= \bE_{:b}+ \bZ_{:k}$ since $b\in G^*_k$. In the above inequality, the difference $||\bE_{:a}|_{2}^2/n - \Gamma_{aa}\big|$ is handled as in the previous case. The second term $ |\langle \bZ_{:k}- \bX_{:ne(a)},\bE_{:a}\rangle/n|\leq \sup_{c \notin G^*_k}|\langle \bZ_{:k}- \bX_{:c},\bE_{:a}\rangle/n|$ is bounded by a supremum of centered quadratic forms of Gaussian variables and is therefore smaller than $c\sqrt{\log(p)/n}|\Gamma|_{\infty}|^{1/2}\Sigma|_{\infty}^{1/2}$ with probability larger than $1-p^{-2}$. The fourth term $\big| \langle \bX_{:a}- \bX_{:ne(a)},\bE_{:b}\rangle/n \big|$ is handled analogously. To control the last quantity $\big| \langle \bX_{:a}- \bX_{:ne(a)},\bX_{:b}\rangle/n \big|$, we use the definition \eqref{eq:definition_ne(a)} 
$$
\big|\big\langle \bX_{:a}-\bX_{:ne(a)}, \frac{\bX_{:b}}{|\bX_{:b}|_2}\big\rangle\big|\leq
\max_{c\neq a, ne(a)}\big|\big\langle \bX_{:a}-\bX_{:ne(a)}, \frac{\bX_{:c}}{|\bX_{:c}|_2}\big\rangle\big|\leq
\max_{c\neq a,b}\big|\big\langle \bX_{:a}-\bX_{:b}, \frac{\bX_{:c}}{|\bX_{:c}|_2}\big\rangle\big|.
$$
Since  $b\in G^*_k$, the random variable $\bX_{:a}-\bX_{:b}= \bE_{:a}-\bE_{:b}$ is  independent from all  $\bX_{:c}$, with $c\neq a,b$. Since $\var{E_{a}-E_{b}}\leq 2|\Gamma|_{\infty}$, we use the Gaussian concentration inequality together with an union bound to get
 \begin{equation*}\label{eq:event}
  \max_{c\neq a,b} \big|\big\langle \bX_{:a}-\bX_{:b}, \frac{\bX_{:c}}{|\bX_{:c}|_2}\big\rangle\big|\leq |\Gamma|^{1/2}_{\infty}\sqrt{12\log(p)}
 \end{equation*}
 with probability larger than $1-p^{-2}$. As a consequence, 
\[
 \big|\big\langle \bX_{:a}-\bX_{:ne(a)}, \frac{\bX_{:b}}{|\bX_{:b}|_2}\big\rangle\big|\leq |\Gamma|^{1/2}_{\infty}\sqrt{12\log(p)},
\]
 with probability larger than $1-p^{-2}$.
Since $|\bX_{:b}|_2^2/\Sigma_{bb}$ follows a $\chi^2$ distribution with $n$ degrees of freedom, we have $|\bX_{:b}|_2\leq c  n^{1/2}|\Sigma|_{\infty}^{1/2}$ with probability larger than $1-p^{-2}$. Putting everything together, we have shown that 
\[
  |\widehat\Gamma_{aa} - \Gamma_{aa}|\leq c |\Gamma|^{1/2}_{\infty}|\Sigma|_{\infty}^{1/2} \sqrt{\frac{\log(p)}{n}}\ ,
\]
with probability larger than $1-c'/p^2$. Taking an union bound over all $a\in [p]$ concludes the proof.

\section{Deviation inequalities}\label{sec:LaurentMassart}

\begin{lem}[Quadratic forms of Gaussian variables~\cite{Laurent00}]\label{lem:quadratic}
Let $Y$ stands for a standard Gaussian vector of size $k$ and let $A$ be a symmetric matrix of size $k$. For any $t>0$, 
\beq\label{eq:quadratic}
 \mathbb{P}\left[Y^t A Y \geq tr(A)+ 2\|A\|_F\sqrt{t}+2 \|A\|_{op}t\right]\leq e^{-t}\ .
\eeq
\end{lem}
Laurent and Massart~\cite{Laurent00} have only stated a specific version of Lemma \ref{lem:quadratic} for positive matrices $A$, but their argument straightforwardly extend to general symmetric matrices $A$.

\end{document}